\documentclass
[
a4paper,
pdftex,
onecolumn,
DIV=11,
abstracton
]
{scrartcl}

\usepackage
{
	graphicx,
	amssymb,
	amsmath,
	amsthm,
	xcolor,
	dsfont,
	algpseudocode,
	authblk,
	enumerate,
	bm,
	mathtools,
	float
}

\usepackage[left=3.5cm,right=3.5cm,top=2cm]{geometry}

\definecolor{upbgray}{rgb}{0.780392156863, 0.788235294118, 0.780392156863}
\definecolor{upbblue}{rgb}{0, 0.125490196078, 0.3568627451}
\definecolor{upbblue2}{rgb}{0, 0.623529412, 0.874509804}

\usepackage[ruled,vlined]{algorithm2e}
\usepackage[bf]{caption}

\DeclareMathAlphabet{\mathpzc}{OT1}{pzc}{m}{it}

\newcommand{\R}{\mathbb{R}}

\newcommand{\cB}{{\mathcal{B}}}

\newcommand{\setdist}[1]{\textrm{dist}\left(#1\right)}
\providecommand{\abs}[1]{\left\lvert #1 \right\rvert}
\providecommand{\norm}[1]{\left\lVert #1 \right\rVert}

\DeclareMathOperator{\diam}{diam}

\newtheorem{theorem}{Theorem}[section]

\newtheorem{lemma}[theorem]{Lemma}
\newtheorem{proposition}[theorem]{Proposition}
\newtheorem{definition}[theorem]{Definition}
\newtheorem{remark}[theorem]{Remark}

\makeatletter
\renewcommand*\env@matrix[1][*\c@MaxMatrixCols c]{%
	\hskip -\arraycolsep
	\let\@ifnextchar\new@ifnextchar
	\array{#1}}
\makeatother

\newcommand{\old}[1]{\textcolor{green}{{ }}}

\begin{document}
\title{A set-oriented path following method for the approximation of parameter dependent attractors}

\author[1]{Raphael Gerlach}
\author[1]{Adrian Ziessler}
\author[2]{Bruno Eckhardt}
\author[1]{Michael Dellnitz}
\affil[1]{\normalsize Department of Mathematics, Paderborn University, 33098 Paderborn, Germany.}
\affil[2]{\normalsize Fachbereich Physik, Philipps-Universit\"{a}t Marburg, 35032 Marburg, Germany.}
	
	\maketitle
	
	\begin{abstract}
		In this work we present a set-oriented path following method for the computation of relative global attractors of parameter-dependent dynamical systems. We start with an initial approximation of the relative global attractor for a fixed parameter $\lambda_0$ computed by a set-oriented subdivision method. By using previously obtained approximations of the parameter-dependent relative global attractor we can track it with respect to a one-dimensional parameter $\lambda > \lambda_0$ without restarting the whole subdivision procedure. We illustrate the feasibility of the set-oriented path following method by exploring the dynamics in models for shear flows during the transition to turbulence.
	\end{abstract}
		
	\section{Introduction}
	
	Over the last two decades so-called \emph{set-oriented numerical methods} have been developed in the context of the numerical treatment of finite dimensional dynamical systems (e.g.,~\cite{DH96,DH97, DJ99, FD03, FLS10}). Here, the basic idea is to cover the objects of interest such as \emph{attractors}, \emph{invariant manifolds} or \emph{almost invariant sets} by outer approximations which are created via \emph{subdivision} and \emph{continuation} techniques. The numerical effort depends essentially on the dimension of the global attractor, i.e., it is easier to compute a one-dimensional attractor in a ten-dimensional space than to compute a three-dimensional attractor in a four-dimensional space \cite{DH97}. The set-oriented techniques have been used successfully in several different application areas such as molecular dynamics (\cite{DDO+99,SHD01,DGM+05}), astrodynamics (\cite{DJLMPPRT05,DJK+05,DJ05}) and ocean dynamics (\cite{FHRSSG12}). Moreover, a set-oriented numerical framework has recently been developed which allows to perform uncertainty quantification for dynamical systems from a global point of view \cite{DKZ17} and the computation of finite dimensional attractors or unstable manifolds of infinite dimensional dynamical systems \cite{DHZ16,ZDG18,Zie18}. In this work we connect these set-oriented algorithms with ideas from bifurcation analysis and path following methods in order to efficiently treat parameter dependencies in the underlying dynamical system.
	
	Systems of physical interest typically depend on parameters appearing in the defining systems of equations, i.e., one considers dynamical systems of the form
	\begin{equation}\label{eq:ODE}
	\dot x = F(x,\lambda),
	\end{equation}
	where $x \in \R^n$ are the variables, $\lambda \in \R^p$ are the system parameters, and $F:\R^n\times \R^p \to \R^n$ is assumed to be sufficiently smooth. By varying $\lambda$, the qualitative structure of the solutions might change significantly, e.g., stable equilibria of the dynamical system might become unstable. Such phenomena are known as \emph{bifurcations} and the parameter values are called \emph{bifurcation parameters} (see~\cite{GH83,HK12,GSS12,Wig13} for a detailed overview). 
	
	Bifurcations typically come in two kinds. On the one hand, one can look at \emph{local bifurcations} \cite{GH83,GSS12}, where the analysis of such bifurcations is generally performed by studying the corresponding vector fields near equilibrium points. That is, one can usually reduce the problem to the analysis of an equation of the form
	\begin{equation}\label{eq:local_f}
	F(x,\lambda) = 0.
	\end{equation}
	Starting with an initial $\lambda_0$ the aim is to find numerically an approximation of a solution curve depending on $\lambda$ which is implicitly defined by 
	\eqref{eq:local_f}. The basic idea of path following methods which tackle this kind of problem for parameter-dependent fixed points and steady states can be found, e.g., in \cite{AG93} and \cite{AG97}. For the numerical treatment of bifurcation problems there exist a wide range of software tools, e.g., the well-known software package \texttt{AUTO} \cite{Doe81}. A first-version of \texttt{AUTO} has already been developed in 1981. More recent contributions to bifurcation analysis software are the \texttt{MATLAB} packages \texttt{MATCONT} \cite{DGK03} and \texttt{COCO} \cite{DS13}.
	
	On the other hand, there are \emph{global bifurcations} \cite{GS03,Wig13} for which it is not sufficient to reduce the study to a neighborhood of an equilibrium or a fixed point. For example, these bifurcations occur when invariant sets, such as periodic orbits, collide with equilibria. Therefore, typical examples are formations of homoclinic and heteroclinic orbits which can also be found numerically with the software package \texttt{AUTO} \cite{Doe81} and \texttt{COCO} by formulating and solving an appropriate boundary value problem (e.g., \cite{DF89}). In addition, \texttt{MATCONT} is able to continue homoclinic and heteroclinic orbits by a homotopy method \cite{WGK+12}.
	
	In this work we extend the subdivision method for the approximation of the so--called \emph{relative global attractor} to parameter-dependent dynamical systems, i.e., we design a path following method from a global point of view, which indirectly allows the numerical analysis of global bifurcations. 
	Note that the relative global attractor contains all backward invariant sets and, in particular, the homoclinic or heteroclinic connections. Using previously obtained approximations of the relative global attractor we can track it with respect to a one-dimensional parameter without restarting the whole subdivision procedure. Since one computation of an attracting set might be very expensive by using the set-oriented method, the algorithm developed in this article can reduce the overall computational cost significantly. The set-oriented path following algorithm has been implemented in the dynamical systems software package \texttt{GAIO} \emph{(Global Analysis of Invariant Objects)} \cite{DFJ01}, which is available for \texttt{MATLAB} (see~https://github.com/gaioguy/GAIO).
	
	We will use the set-oriented path following method to explore the dynamics of shear flows during the transition to turbulence. In contrast to convection or Taylor-Couette flows, where the transition can be captured by sequences of bifurcations that branch off the laminar state \cite{Chandrasekhar61,Chossat94}, shear flows such as pipe flow or plane Couette flow do not show a linear instability of the laminar profile \cite{Schmid99,Grossmann:2000}. The transition is then mediated by the presence of 3d states that form in saddle-node bifurcations \cite{Nagata:1990ae,Eckhardt:2007ka,Eckhardt:2007ix,Eckhardt:2018js}. Very close to the bifurcation, the states of the upper branches can be stable, but they soon undergo secondary bifurcations and period doubling cascades \cite{Kreilos:2012bd,Avila:2013jq,Zammert:2015jg}. Crisis bifurcations destroy the attractors and give rise to the open saddles that underlie the observed transience of the turbulent state \cite{Hof:2006ab,Kreilos:2014ew}.
	
	Exploring the dynamics in the full phase space is very complicated because of the high-dimensionality of the relevant subspaces, though notable progress has been made in some cases \cite{Halcrow:2009jx,Gibson:2009kp}. In order to prepare for the application to such computationally expensive examples, we begin here with an analysis of simple models that capture much of the phenomenology of shear flows: they have a laminar state that is linearly stable and a saddle-node bifurcation that represents the transition to turbulence. The models even contain the transition from an attractor to a transient saddle.
	
	The simplest model where the modes have a physical interpretation is the four-dimensional model proposed by Waleffe \cite{Wal95a}: the four components represent the transversal velocity components, the vortices, the streak, and the mean velocity. With the methods described here, we are able to trace the appearance of the attractor and its growth as the Reynolds number increases. We will then extend the analysis to the case of the nine-dimensional model discussed in \cite{MFE04,MFE05}.
	
	A detailed outline of the paper is as follows: in Section~\ref{sec:review} we briefly summarize the subdivision method introduced in \cite{DH97}. In Section~\ref{sec:path} we then develop the set-oriented path following method which allows us to compute attracting sets of parameter-dependent dynamical systems without restarting the subdivision procedure. In Section~\ref{sec:numerical_examples} we illustrate the efficiency of the path following method for two different reduced order models from fluid dynamics, namely, a four-dimensional model of self-sustained flows and a nine-dimensional model of turbulent shear flows. 
	Conclusions are given in Section~\ref{sec:conclusion}.
	
	\section{A set-oriented path following method}\label{sec:rel_A}
	We consider parameter dependent dynamical systems of the form
	\begin{equation}\label{eq:DS}
	x_{j+1} = f(x_j, \lambda),\quad j=0,1,\ldots,
	\end{equation}
	where $x_j \in \R^n$ and $f:\R^n \times \Lambda \rightarrow \R^n$ is a homeomorphism for each parameter $\lambda$ in an open subset $\Lambda \subset \R$ and uniformly continuous in $\lambda$ on bounded sets, e.g., $f$ could be a time-$T$-map for the system \eqref{eq:ODE}. In what follows, for fixed $\lambda \in \Lambda$ we use the abbreviation $f_\lambda(x):=f(x, \lambda)$. Moreover, we assume that $f_{\lambda}$ has a compact global attractor $A^\lambda\subset \R^n$ for each $\lambda\in \Lambda$, that is, the compact set $A^\lambda\subset \R^n$ attracts any bounded set $B\subset \R^n$.
	Later on, we will additionally assume that $A^\lambda$ is upper semi-continuous in $\lambda$. 
	
	The aim of this work is to develop a set-oriented path following method for the approximation of $A^\lambda$ depending on the parameter $\lambda$.
	Since this  method is based on the framework developed in \cite{DH97} we start with a short review of the related material.
	
	\subsection{Brief review of the subdivision method}\label{sec:review}
	Let $Q \subset \R^k$ be a compact set. For a homeomorphism $g:\R^n\to \R^n$ we define the \emph{global attractor relative to} $Q$ by
	\begin{equation}
	\label{eq:relativeAttractor}
	A_Q = \bigcap_{j\ge 0} g^j(Q).
	\end{equation}
	Observe that $A_Q=A$ if $A \subset Q$ is the compact global attractor of $g$.
	
	As a first step we approximate this set with a subdivision procedure introduced in \cite{DH97} for a fixed $\lambda_0\in\Lambda$, that is, $g=f_{\lambda_0}$.
	Roughly speaking, the idea of the algorithm is to start with a finite family  $\cB_0$ of compact subsets of $\R^n$ that are a partition of $Q$.
	Then we refine this partition and remove those subsets that do not contain parts of the relative global attractor. By continuing this process we generate a sequence $\cB_0,\cB_1,\ldots$ of finite collections of compact subsets of $\R^n$ such that the diameter
	\[
	\diam(\cB_k) = \max_{B\in\cB_k}\diam(B)
	\]
	converges to zero and the union $\bigcup_{B\in B_k} B$ approaches the relative global attractor $A_Q$ for $k\rightarrow\infty$. 
	
	More precisely, let $\cB_0$ be the initial collection, e.g., $\cB_0 = \{Q\}$. Then we inductively obtain $\cB_k$ from $\cB_{k-1}$
	for $k=1,2,\ldots$ in two steps.
	
	\begin{enumerate}
		\item {\em Subdivision:} Construct a new collection
		$\widehat\cB_k$ such that
		\begin{equation*}
			\label{eq:sd1}
			\bigcup\limits_{B\in\widehat\cB_k}B = \bigcup\limits_{B\in\cB_{k-1}}B
		\end{equation*}
		and
		\begin{equation*}
			\label{eq:sd2}
			\diam(\widehat\cB_k) = \theta\diam(\cB_{k-1}),
		\end{equation*}
		where $0<\theta< 1$.
		\item {\em Selection:} Define the new collection $\cB_k$ by
		\begin{equation*}
			\label{eq:select}
			\cB_k=\left\{B\in\widehat\cB_k : \exists \widehat B\in\widehat\cB_k
			~\mbox{such that}~g^{-1}(B)\cap\widehat B\ne\emptyset\right\}.
		\end{equation*}
	\end{enumerate}
	The first step guarantees that the collections $\cB_k$ consist of
	successively finer sets for increasing $k$.  In fact, by construction
	\begin{equation*}\label{eq:diamB}
		\diam(\cB_k)\leq\theta^k\diam(\cB_0)\rightarrow 0\quad
		\mbox{for $k\rightarrow\infty$.}
	\end{equation*}

	In the second step we remove each subset whose preimage does neither
	intersect itself nor any other subset in $\widehat\cB_k$. 
	
	\begin{remark}
		In the application of the subdivision scheme for the computation of the relative global attractor we have to perform the selection step
		\begin{equation*}
			\cB_k=\left\{B\in\widehat\cB_k : \exists \widehat B\in\widehat\cB_k
			~\mbox{such that}~g^{-1}(B)\cap\widehat B\ne\emptyset\right\}.
		\end{equation*}
		Thus, we have to decide whether or not the preimage of a given set $B_i \in \cB_k$ has a nonzero intersection with another set $B_j \in \widehat{\cB_k}$, i.e.,
		\begin{align}\label{eq:nonzeroIntersectionBoxes}
			g^{-1}(B_i)\cap\widehat B_j\ne\emptyset.
		\end{align}
		Numerically this is realized as follows: we discretize each box $\widehat{B}_j$ by a finite set of test points $x \in \widehat{B}_j$ and replace the condition \eqref{eq:nonzeroIntersectionBoxes} by
		\begin{align*}
			g(x) \notin B_i \quad \mbox{for all test points } x \in \widehat{B}_j.
		\end{align*}
	\end{remark}
	
	The selection step is responsible for the fact that the unions
	$\bigcup_{B\in\cB_k}B$ approach the relative global attractor.
	Denote by $Q_k$ the collection of compact sets obtained after $k$ subdivision steps, that is
	\begin{equation}\label{eq:Qk}
	Q_k=\bigcup_{B\in \cB_k} B_.
	\end{equation}
	
	Note that these sets form a nested sequence $Q_{k+1}\subset Q_k$. Therefore the limit
	\[
	Q_\infty = \bigcap_{k=0}^\infty Q_k
	\]
	exists and the following result has been proven in \cite{DH97}:
	
	\begin{proposition}\label{prop:convergence_subdivision}
		Let $A_Q$ be the global attractor relative to $Q$, and let $\cB_0$ be a finite collection of closed subset with $Q_0=\bigcup_{B\in \cB_0} =Q$. 
		Then the sets $Q_k$ obtained by the subdivision algorithm contain the relative global attractor, $A_Q\subset Q_k$ for all $k=0,1,\ldots$, and moreover
		\[
		A_Q=Q_\infty.
		\]
	\end{proposition}
	
	\subsection{Path following method for the approximation of relative global attractors}\label{sec:path}
	In this section we develop a path following algorithm that allows to compute the relative global attractor for various parameter values $\lambda \in \Lambda$ of \eqref{eq:DS} by reusing previously obtained coverings. The idea of this method is to first approximate the relative global attractor $A_Q^{\lambda_0}$ for $\lambda_0 \in \Lambda$ and then to use a covering of this set, denoted by $Q_k^{\lambda_0}$, as a starting point to compute the relative global attractor $A_Q^{\lambda_1}$ for $\lambda_1\in \Lambda$ sufficiently close to $\lambda_0$.
	
	We start with the following observation.
	\begin{lemma}\label{lem:A_equal_to_AQ_for_arbitrary_Q}
		Let $f:\R^n \to \R^n$ be a homeomorphism and assume that $f$ has a compact global attractor $A \subset \R^n$. Furthermore, let $Q^1, Q^2$ be compact sets such that $A \subset Q^1 \subset Q^2$. Then,
		\begin{equation*}
			A_{Q^1} = A = A_{Q^2},
		\end{equation*}
		where $A_{Q^i}$ is the global attractor relative to $Q^i$ for $i=1,2$.
	\end{lemma}
	\begin{proof}
		This statement follows directly by the definition of the global attractor relative to a compact set $Q$ (cf.~\eqref{eq:relativeAttractor}) and the fact that $A \subset Q^1 \subset Q^2$.
	\end{proof}
	
	As a direct consequence we obtain the following result.
	\begin{proposition}\label{prop:convergence_boxCoverings}
		Let $f:\R^n \to \R^n$ be a homeomorphism and assume that $f$ has a compact global attractor $A \subset \R^n$. Let $Q^1, Q^2$ be compact sets such that $A \subset Q^1 \subset Q^2$ and denote by $Q_k^i$, $i = 1,2$, the corresponding sets obtained after $k$ subdivision and selection steps (see~\eqref{eq:Qk}). Then
		\begin{equation*}
			\lim_{k \to \infty} h(Q_k^1, Q_k^2) = 0,
		\end{equation*}
		where $h(B,C)$ denotes the usual Hausdorff distance between two compact subsets $B,C \subset \R^n$.
	\end{proposition}
	\begin{proof}
		This result is a direct consequence of Proposition~\ref{prop:convergence_subdivision} and Lemma~\ref{lem:A_equal_to_AQ_for_arbitrary_Q}.
	\end{proof}
	
	Proposition~\ref{prop:convergence_boxCoverings} states that by performing sufficiently many subdivision and selection steps we obtain a good approximation of the attracting set $A$ no matter how we choose the initial covering $Q$ as long as $A \subset Q$.
	
	In what follows, let us assume that $Q\subset \R^n$ is a compact set such that $A^\lambda\subset Q$ for all $\lambda\in \Lambda$. We denote by $\mathcal{Q}(\lambda_0)$ a finite family of compact sets $Q_k^{\lambda_0}$, $k=0,1,\ldots,m$, which all contain the global attractor $A_Q^{\lambda_0}$ relative to $Q$ of the map $f_{\lambda_0}$. Next, either we fix $\lambda_1\in \Lambda$ and choose $K \leq m$ sufficiently small such that $A^{\lambda_1} \subset Q_K^{\lambda_0}$ or we fix $K\leq m$ and choose $\lambda_1 \in \Lambda$ sufficiently close to $\lambda_0$ such that $A^{\lambda_1} \subset Q_K^{\lambda_0}$. In any case Lemma~\ref{lem:A_equal_to_AQ_for_arbitrary_Q} and Proposition~\ref{prop:convergence_boxCoverings} allows us to approximate the global attractor $A_Q^{\lambda_1}$ of the map $f_{\lambda_1}$ using the initial covering $Q_K^{\lambda_0}$ with the subdivision algorithm described in the previous section. Observe that we can always choose $K=0$ for every $\lambda_1$, since $Q_0^{\lambda_0} = Q$ and $A_Q^{\lambda_1} \subset Q$ for $\lambda_1 \in \Lambda$ by assumption. 
	
	To discuss the choice of $\lambda_1$ for fixed $K$ we have to analyze the change of the attractor $A^{\lambda}$ when the parameter $\lambda$ is varied. To this end, we define the distance between two subsets $B,C$ of $\R^n$ as
	\[
	\setdist{B,C}=\sup_{x\in B}\inf_{y\in C} \norm{x-y}.
	\]
	This distance allows us to define the continuity of attractors as follows:
	
	\begin{definition}
		Let $\lambda_0\in \Lambda$. The family of attractors $A^\lambda$ is \emph{upper semi-continuous} at $\lambda_0$ if
		\begin{equation}\label{def:upper-semi-cont}
		\lim_{\lambda\to \lambda_0}\setdist{A^\lambda,A^{\lambda_0}}=0.
		\end{equation}
		The attractor $A^\lambda$ is called \emph{upper semi-continuous} if $A^\lambda$ is upper semi-continuous for each $\lambda\in \Lambda$.
	\end{definition}
	Note that upper semi-continuity at $\lambda_0$ implies that for every $\varepsilon>0$ there exists a neighborhood $U_\delta(\lambda_0)\subset  \Lambda$ of $\lambda_0$ such that $A^\lambda\subset U_\varepsilon(A^{\lambda_0})$ for all $\lambda \in U_\delta(\lambda_0)$, where $U_\varepsilon(A^{\lambda_0})\subset \R^n$ denotes the $\varepsilon$-neighborhood of $A^{\lambda_0}$. Thus, the attractor cannot become larger instantaneously by varying $\lambda$ which is a naturally needed property for our proposed path following scheme. In fact, if the attracting set suddenly ``explodes'' we cannot assure that any previously computed covering $Q_k^{\lambda}$ besides $k=0$ still covers $A^{\lambda_1}$. For the sake of completeness we note that \emph{lower semi-continuity} of $A^\lambda$, i.e.,
	\begin{equation*}
		\lim_{\lambda\to \lambda_0}\setdist{A^{\lambda_0},A^{\lambda}}=0 \quad \forall \lambda\in\Lambda.
	\end{equation*}
	prevents a sudden shrinking of $A^\lambda$ which holds for instance for gradient systems with hyperbolic fixed points \cite{HG89}. However, this would be, in general, no issue for our proposed algorithm.
	
	We are now in the position to prove that for each $K\in\{0,\ldots,m\}$ there is a range of parameters $\lambda\in \Lambda$ such that the attractor $A^\lambda$ can be approximated by the subdivision algorithm with the initial compact set $Q_K^{\lambda_0}$. 
	
	\begin{proposition}\label{prop:upper-semi}
		Let $\mathcal{Q}(\lambda_0)=\{Q_0^{\lambda_0},\ldots,Q_m^{\lambda_0}\}$ be the family of sets generated by the subdivision method
		such that $A_Q^{\lambda_0} = A^{\lambda_0}\subset Q_k^{\lambda_0}$ for all $k=0,\ldots,m$. Suppose that there is $\varepsilon>0$ such that $U_\varepsilon(A^{\lambda_0})\subset Q_k^{\lambda_0}$ for all $k=0,\ldots,m$ and $A^{\lambda}$ is upper semi-continuous at $\lambda_0$.
		Then for every $k=0,\ldots,m$ there exists $\delta=\delta(k)>0$ such that $A^\lambda\subset Q_k^{\lambda_0}$ for all $\lambda \in U_\delta(\lambda_0)$. In particular, $A^\lambda$ can be approximated by using the initial compact set $Q_k^{\lambda_0}$.
	\end{proposition}
	
	\begin{proof}
		Let $k\in \{0,\ldots,m\}$ be fixed. Due to the upper semi-continuity of $A^\lambda$ in $\lambda_0$ there is $\delta>0$ such that 
		\[
		\setdist{A^\lambda,A^{\lambda_0}} < \varepsilon \qquad\forall \lambda\in U_\delta(\lambda_0).
		\]
		Thus, we conclude by assumption that
		\[
		A^{\lambda}\subset U_\varepsilon(A^{\lambda_0}) \subset Q_k^{\lambda_0} \qquad\forall \lambda\in U_\delta(\lambda_0).
		\]
	\end{proof}
	
	In the following we will call such $\lambda \in U_\delta(\lambda_0)$ \emph{feasible} and Proposition~\ref{prop:upper-semi} guarantees the existence of a feasible $\lambda\in \Lambda$. Throughout the remainder of this article we now suppose that $A^{\lambda}$ is upper semi-continuous and the additional assumptions of Proposition~\ref{prop:upper-semi} hold. Then the numerical realization of the set-oriented path following algorithm can be described as follows:
	
	\begin{algorithm}[H]
		\caption{Set-oriented path following method}
		\vspace{1em}
		\flushleft{\em Initialization:} Let $Q\subset \R^n$ be compact such that $A^{\lambda}\subset Q$ for all $\lambda\in \Lambda$. Fix $\lambda_0\in \Lambda$ and use the subdivision method described in section~\ref{sec:review} to obtain a family $\mathcal{Q}(\lambda_0)=\{Q_0^{\lambda_0},\ldots,Q_{m}^{\lambda_0}\}$ of approximations of $A^{\lambda_0}$.\\
		\vspace{1em}
		\flushleft{\em Path following:} For $j = 0,1,\ldots$
		\begin{enumerate}
			\item Choose $K_{j+1}\in \{0,\ldots,m\}$ and take a feasible $\lambda_{j+1}>\lambda_j$ according to Proposition~\ref{prop:upper-semi}
			such that $A^{\lambda_{j+1}}\subset Q_{K_{j+1}}^{\lambda_j}$.
			\item Perform $m-K_{j+1}$ subdivision and selection steps for $f_{\lambda_{j+1}}$ on $Q_{K_{j+1}}^{\lambda_j}$ in order to generate a new family
			\[
			\mathcal{Q}(\lambda_{j+1})=\{Q_0^{\lambda_{j+1}},\ldots,Q_{m}^{\lambda_{j+1}}\}
			\]
			of approximations of $A^{\lambda_{j+1}}$, where $Q_k^{\lambda_{j+1}}=Q_k^{\lambda_j}$ for $k=0,\ldots,K_{j+1}$.
		\end{enumerate}
		\label{alg:pathFollowing}
	\end{algorithm}
	
	In Figure~\ref{fig:pathFollowing}, we illustrate two steps of Algorithm~\ref{alg:pathFollowing} for the Lorenz system \cite{Lo63} given by
	\begin{equation}\label{eq:lor63}
	\begin{aligned}
	\dot x &= \sigma (y-x),\\
	\dot y &= x(\rho - z) - y,\\
	\dot z &= xy - \beta z,
	\end{aligned}
	\end{equation}
	where we use $\beta$ as our parameter of interest. We note that the derivation of the Lorenz model shows that $\beta$ is related to the aspect ratio of the convection cells.
	
	\begin{figure}[!h]
		\begin{center}
			\begin{minipage}{0.320\textwidth}
				\includegraphics[width = \textwidth]{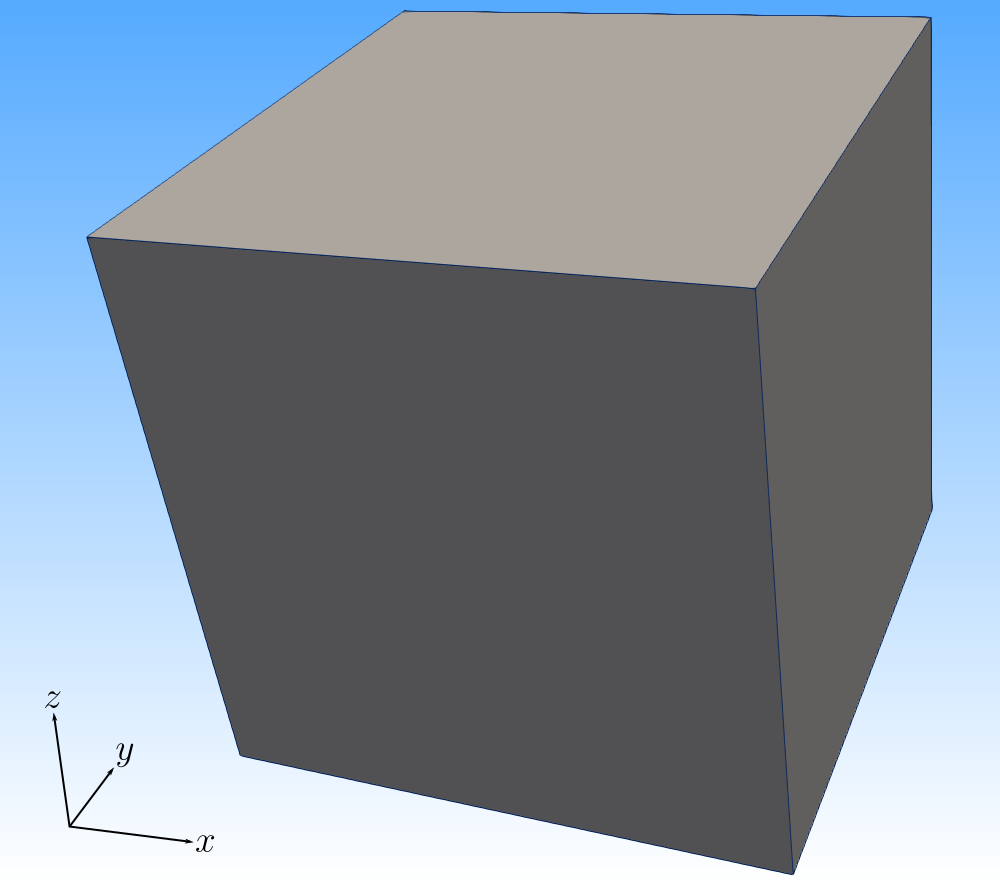}\\
				\centering \scriptsize{(a)}
			\end{minipage}
			\hfill
			\begin{minipage}{0.320\textwidth}
				\includegraphics[width = \textwidth]{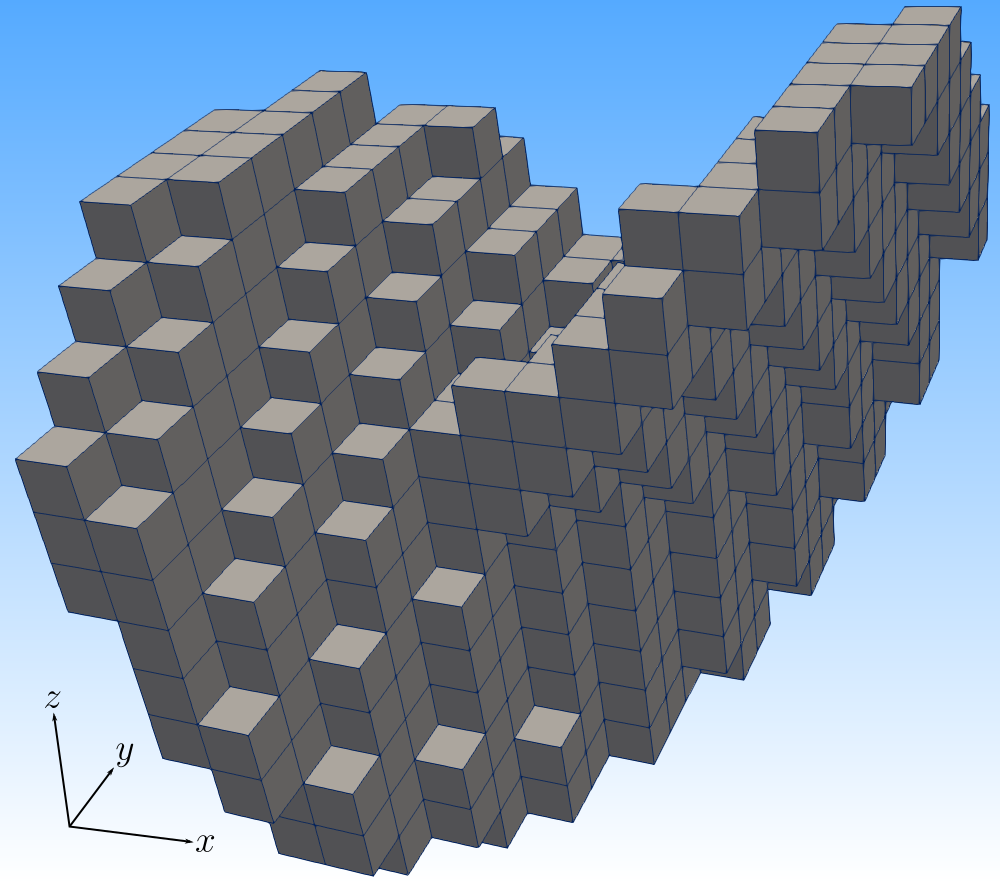}\\
				\centering \scriptsize{(b)}
			\end{minipage}
			\hfill
			\begin{minipage}{0.320\textwidth}
				\includegraphics[width = \textwidth]{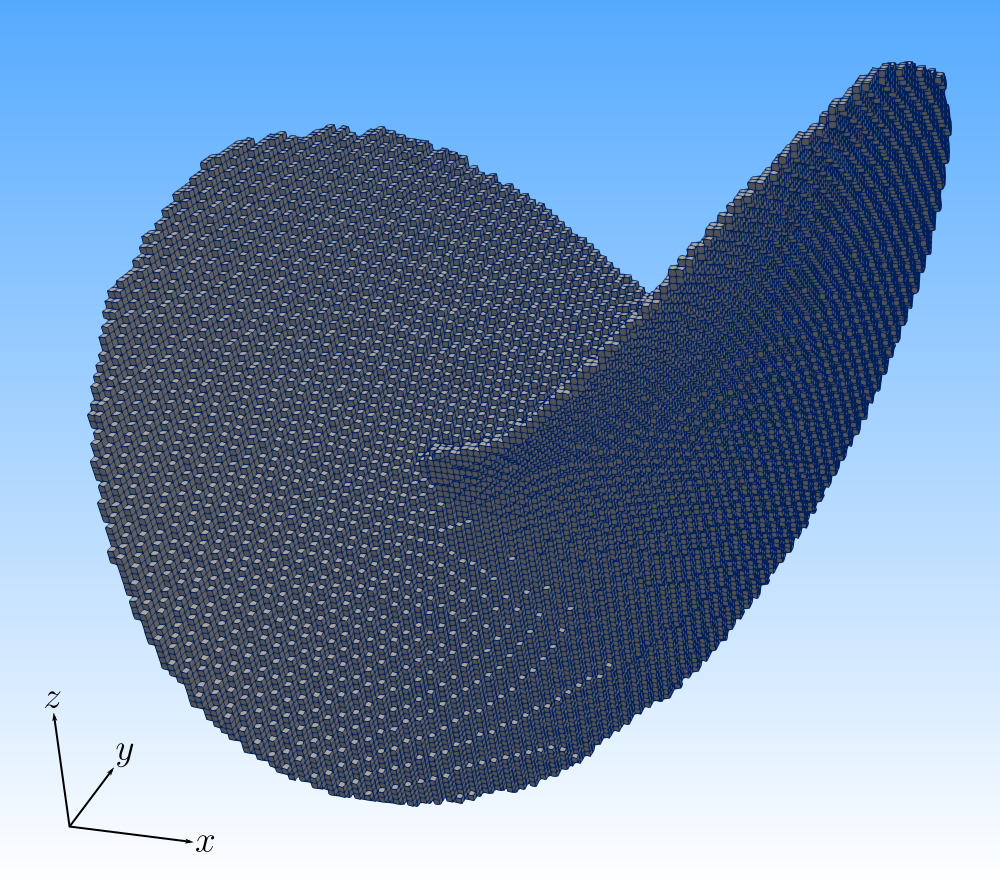}\\
				\centering \scriptsize{(c)}
			\end{minipage}\\
			\begin{minipage}{0.320\textwidth}
				\includegraphics[width = \textwidth]{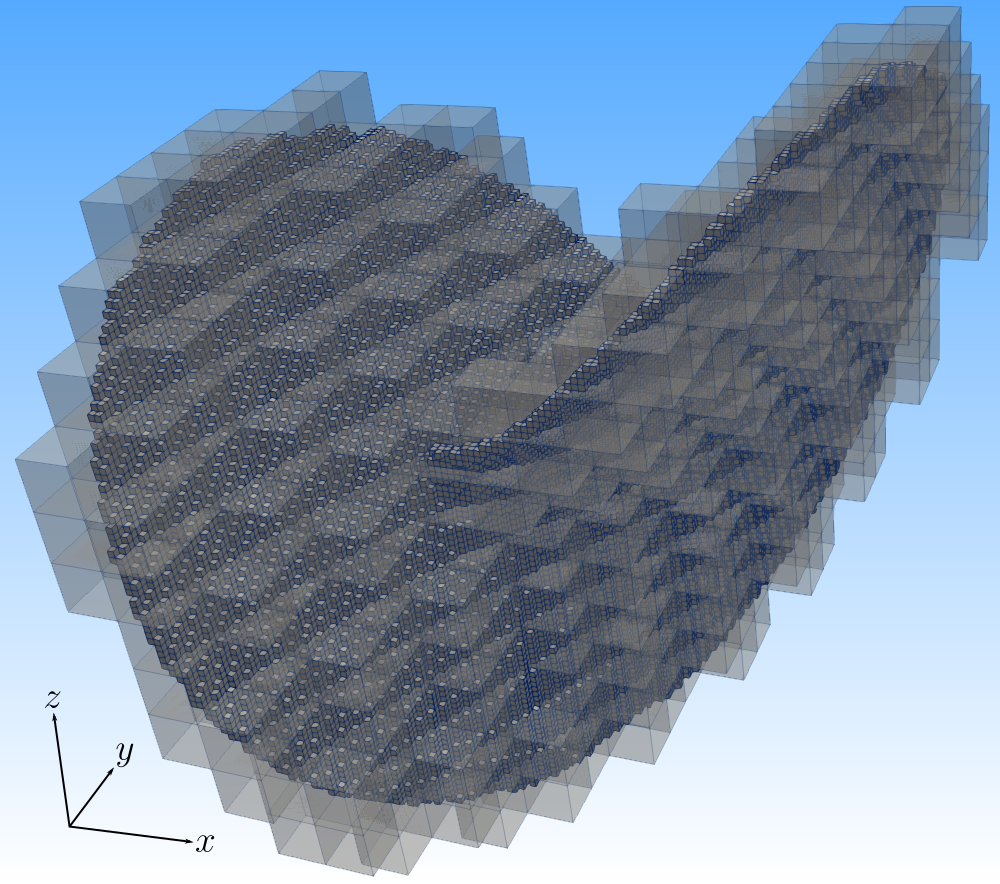}\\
				\centering \scriptsize{(d)}
			\end{minipage}
			\hfill
			\begin{minipage}{0.320\textwidth}
				\includegraphics[width = \textwidth]{Pictures/PF25.png}\\
				\centering \scriptsize{(e)}
			\end{minipage}
			\hfill
			\begin{minipage}{0.320\textwidth}
				\includegraphics[width = \textwidth]{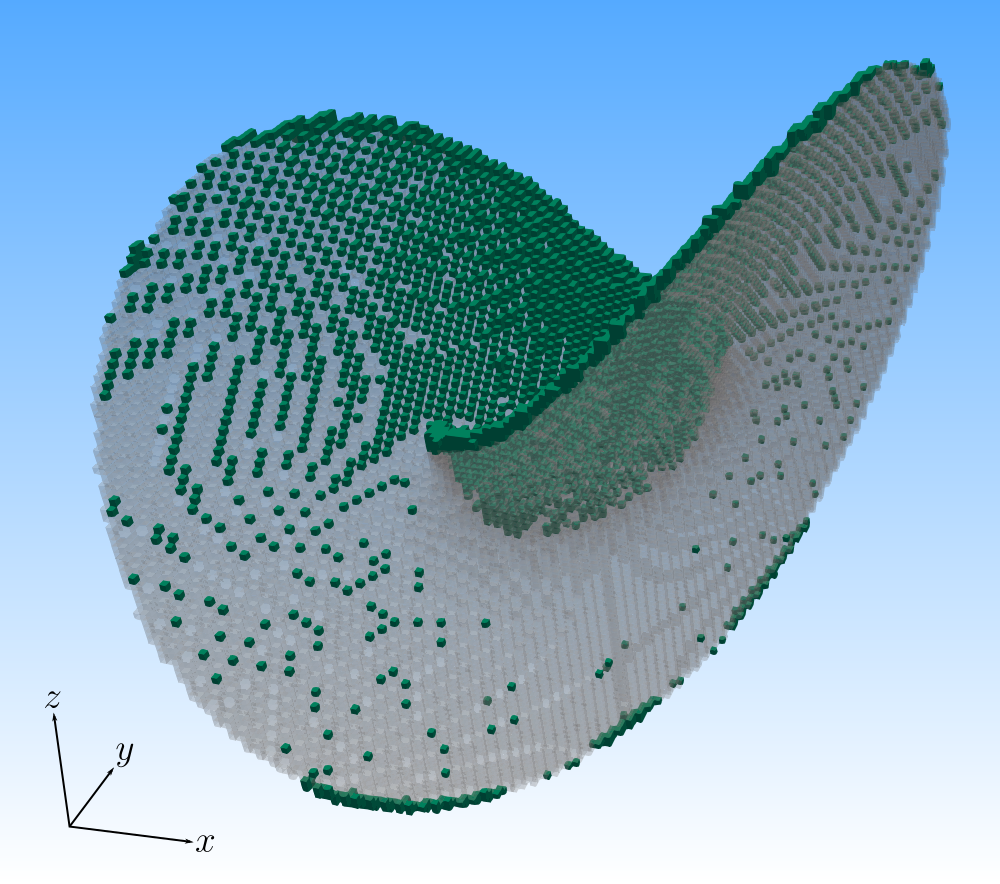}\\
				\centering \scriptsize{(f)}
			\end{minipage}
		\end{center}
		\caption{Illustration of Algorithm~\ref{alg:pathFollowing} for the Lorenz system \eqref{eq:lor63} with parameters $\sigma = 10, \rho = 28$, $\beta_1 = 8/3$ and $\beta_2 = 2.5$. 
			(a) Initial box $Q$. (b) and (c) Successively finer box coverings of the attractor $A_Q^{\beta_1}$, with the width of the boxes decreasing by factors $1/16$ and $1/128$, respectively. The final covering $Q_{21}^{\beta_1}$ in (c) has been obtained after $21$ subdivision and selection steps. 
			(d) Neighborhood $Q_{12}^{\beta_1}$ (transparent boxes) such that $A_Q^{\beta_2} \subset Q_{12}^{\beta_1}$. (e) Initial box covering $Q_{12}^{\beta_1}$ for the approximation of $A_Q^{\beta_2}$. (f) Box covering $Q_{21}^{\beta_2}$ of $A_Q^{\beta_2}$ after $9$ subdivision and selection steps. Here, the solid green boxes illustrate the change of the attracting set.}
		\label{fig:pathFollowing}
	\end{figure}
	
	\begin{remark}\label{rmk:pathFollowing}\quad
		\begin{enumerate}
			\item[(a)] Instead of following the path along $\lambda$ in positive direction one can also take $\lambda_{j+1}<\lambda_j$ and follow the path in the negative direction. 
			\item[(b)] Intuitively, choosing a larger $K_{j+1}$ decreases the range of feasible parameters $\lambda_{j+1}$. 
			However, one has to perform less subdivision and selection steps to reach the same level, i.e., the same fineness of the final box covering of $A^{\lambda_{j+1}}$.
			With more knowledge on the upper semi-continuity property \eqref{def:upper-semi-cont} we can make this more precise in the following Proposition. 
			\item[(c)] If $\lambda_{j+1}$ is not feasible, i.e., $\lambda_{j+1}\notin U_{\delta}(\lambda_j)$,
			$Q_{K_j}^{\lambda_j}$ might not contain all of $A^{\lambda_{j+1}}$ and thus, the subdivision algorithm does not approximate the whole attractor $A^{\lambda_{j+1}}$.
			However, to make the algorithm more robust, we reintroduce sets if an image $f_{\lambda_{j+1}}(x)$ of a test point $x\in \R^n$ is not contained in any set of the current collection $\cB_k(\lambda_j)$. 
		\end{enumerate}
	\end{remark}
	
	After discussing the choice of $\lambda_1\in \Lambda$ for fixed $K\in \{0,\ldots,m\}$ the following Proposition tells us, how to choose $K_{j+1}$ for a fixed step size in the parameter space $\Lambda$, i.e., how to choose $K\in\{0,\ldots,m\}$ when $\lambda_1\in\Lambda$ is given.
	
	\begin{proposition}\label{prop:A_usc}
		Let $\delta>0$ and $\lambda_1\in \Lambda$ with $\abs{\lambda_1-\lambda_0}\leq\delta$. Suppose there is a constant $C>0$ such that
		\begin{align}\label{eq:lin_cont}
			\setdist{A^{\lambda_1},A^{\lambda_0}}\leq C \delta
		\end{align}
		and $U_{C\delta}(A^{\lambda_0}) \subset Q_K^{\lambda_0}$ for one $K\in \{0,\ldots, m\}$. Then $A^{\lambda_1}\subset Q_K^{\lambda_0}$ and, in particular, $\lambda_1$ is feasible for $K$.
	\end{proposition}
	\begin{proof}
		According to \eqref{eq:lin_cont} the $C\delta$-neighborhood $U_{C\delta}(A^{\lambda_0})$ of $A^{\lambda_0}$ contains $A^{\lambda_{1}}$ and by assumption we immediately obtain
		\[
		A^{\lambda_{1}}\subset U_{C\delta}(A^{\lambda_0}) \subset Q_K^{\lambda_0}
		\]
		as claimed.
	\end{proof}
	
	\section{Numerical examples}\label{sec:numerical_examples}
	
	In this section we present results of computations carried out for a 
	four- and a nine-dimensional reduced order model from fluid dynamics.

	\subsection{A four-dimensional model of self-sustained flows}
	As a first example, we consider a four-dimensional nonlinear model of self-sustained flows introduced in 
	\cite{Wal95a} and further discussed in \cite{Wal95b}. The four variables of the model 
	represent the spanwise velocity components ($u$),  the vortices ($v$), the streak ($w$), and the mean profile ($m$) 
	(cf. \cite{Wal95b} for more explanations of the model). Their dynamics is given by
	\begin{equation}\label{eq:waleffe}
	\begin{aligned}
	\frac{\mbox{d}}{\mbox{d}t}\begin{pmatrix} u \\ v \\ w \\ m		\end{pmatrix} = \frac{1}{R} \begin{pmatrix}	0 \\ 0 \\ 0 \\ \sigma 	\end{pmatrix} - \frac{1}{R} \begin{pmatrix} \lambda u\\ \mu v\\ \nu w\\ \sigma m \end{pmatrix} + \begin{pmatrix} 0 & 0 & -\gamma w & v\\ 0 & 0 & \delta w & 0\\ \gamma w & -\delta w & 0 & 0 \\ -v & 0 & 0 & 0 \end{pmatrix}\begin{pmatrix} u \\ v \\ w \\ m \end{pmatrix},
	\end{aligned}
	\end{equation}
	where $R > 0$ is the Reynolds number, and $\lambda$, $\mu$, $\nu$, $\sigma, \gamma$, and $\delta$ are positive
	parameters.
	For the parameters, we take values $\lambda=\mu=\sigma=10$, $\nu=15$, $\delta = 1$ and $\gamma = 0.5$; a saddle-node bifurcation then
	appears for $R_c=98.6325$. 
	
	Our goal is to numerically analyze how the attracting sets change for 
	the Reynolds numbers $R \in [98,400]$. We define $x = (u,v,w,m)^\top$ and denote by $f(x,R)$ the time-$T$-map of \eqref{eq:waleffe} which depends on the Reynolds number $R$ and the integration time $T$. For the following analysis we fix the integration time of $T=20$.
	Furthermore, we choose the initial box $Q = [-0.9, 1.1]\times [-0.8, 1.2] \times [-1,1] \times [-0.8,1.2]$ in which we want to approximate the relative global attractor $A_Q^R$. According to Algorithm~\ref{alg:pathFollowing} we fix $m=36$ and $K=32$. Following of Remark~\ref{rmk:pathFollowing} (a) we start the path following algorithm for the Reynolds number $R_0 = 400$ and define $R_{j+1} = R_j -1$ for $j=0,\ldots,302$ such that the interval $\Lambda=[98,400]$ is discretized equidistantly. A detailed analysis of \eqref{eq:waleffe} for $R \in \Lambda$ can be found in \cite{Wal95b}. In what follows, we will study the attractor from a global point of view. To this end, Figure~\ref{fig:waleffe4Da} and Figure~\ref{fig:waleffe4Db} show three-dimensional projections of the relative global attractor obtained by Algorithm~\ref{alg:pathFollowing} for different Reynolds numbers.
	
	
	For a parameter value just below the saddle-node bifurcation, i.e., for $R = 98$, the attracting set does not contain any invariant structures besides the laminar profile $(u,v,w,m) = (0,0,0,1)$. For $R$ just above the saddle node bifurcation, i.e., for $R = 98.6325$, Figure~\ref{fig:waleffe4Da} (a) shows that the attractor now contains an upper branch steady solution, as well as the lower branch saddle state. This situation persists up to $R = 100.0232$, after which the upper branch undergoes a supercritical Hopf bifurcation.
	
	\begin{figure}[!htb]
		\centering
		\begin{minipage}{0.320\textwidth}
			\includegraphics[width = \textwidth]{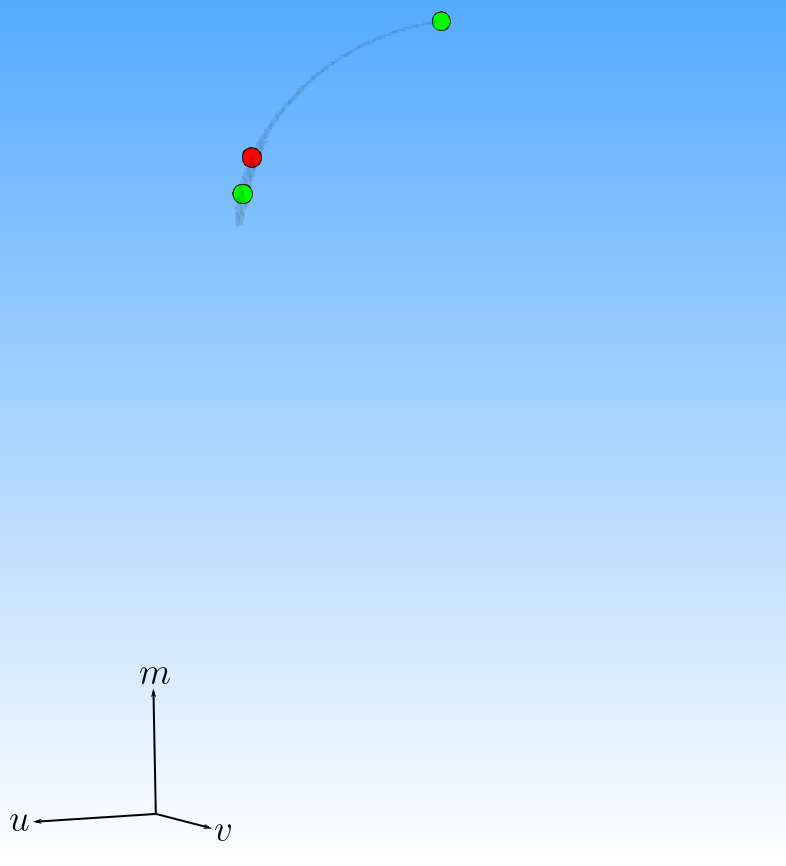}\\
			\centering \scriptsize{(a) $R = 99$}
		\end{minipage}
		\hfill
		\begin{minipage}{0.320\textwidth}
			\includegraphics[width = \textwidth]{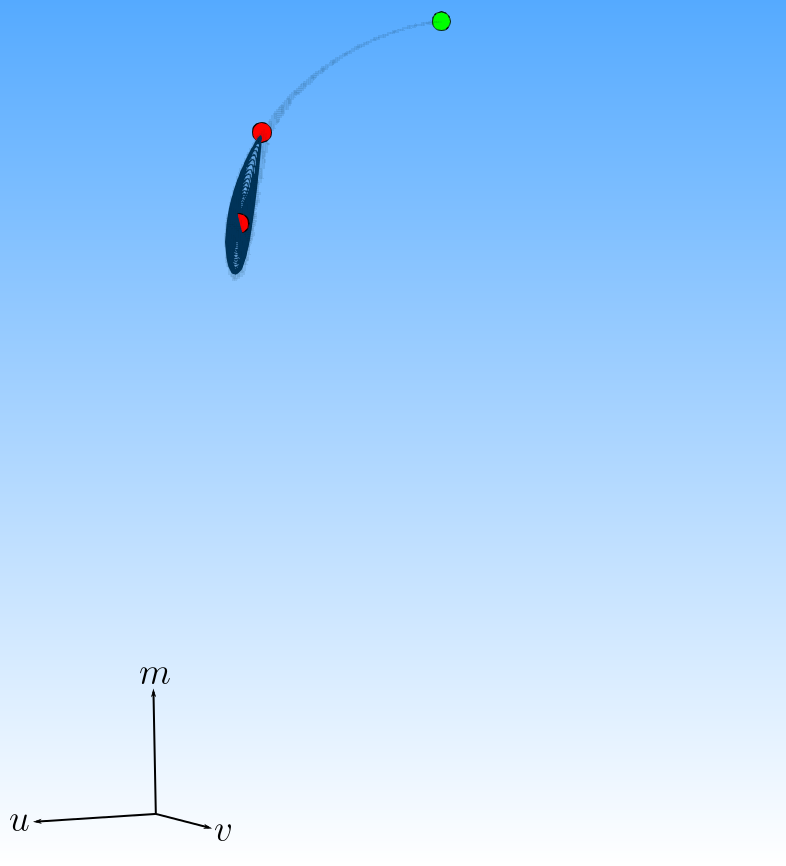}\\
			\centering \scriptsize{(b) $R = 101$}
		\end{minipage}
		\hfill
		\begin{minipage}{0.320\textwidth}
			\includegraphics[width = \textwidth]{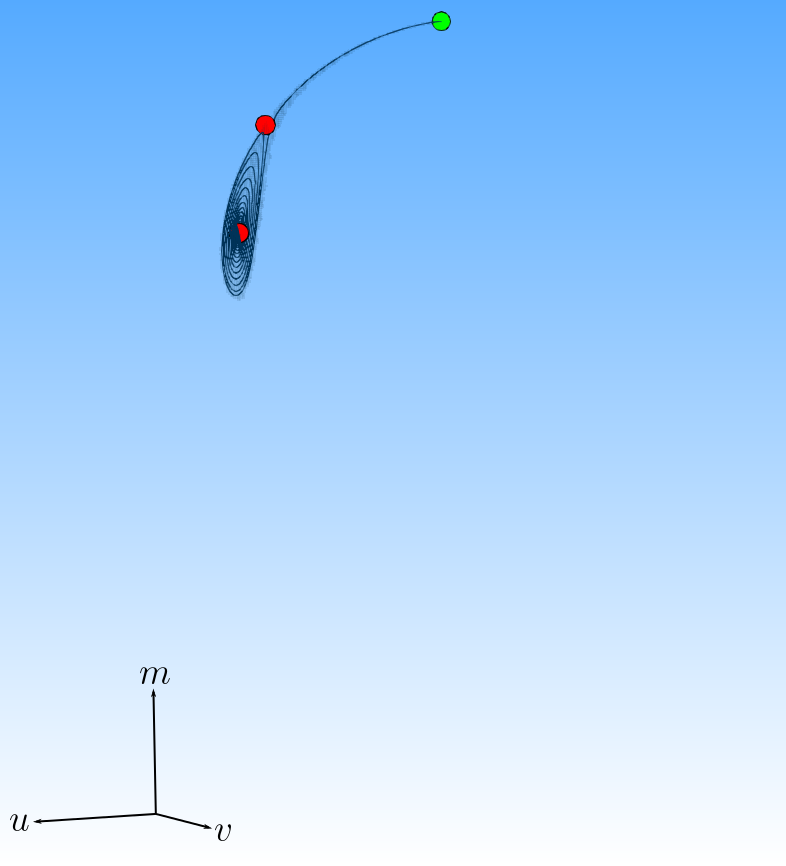}\\
			\centering \scriptsize{(c) $R = 102$}
		\end{minipage}
		\caption{Three-dimensional projections of the relative global attractor $A_Q^{R}$ for the low dimensional model of a self-sustained flow \eqref{eq:waleffe}. Stable and unstable fixed points are shown in green and red, respectively. The green fixed point at the top of each figure depicts the laminar solution $(u,v,w,m) = (0,0,0,1)$ which remains stable for all values of $R$ (cf.~\cite{Wal95b}).}
		\label{fig:waleffe4Da}
	\end{figure}
	
	The limit cycle is stable (cf.~Figure~\ref{fig:waleffe4Da} (b)) until $R = 101.0311$ (cf.~Figure~\ref{fig:waleffe4Da} (c)), where it disappears in a homoclinic bifurcation. Note that the attracting set does not change significantly, although the dynamics does: in the homoclinic bifurcation
	the attractor rips open and becomes a transient saddle. In order to demonstrate 
	this change in the dynamics, we compare in Figure~\ref{fig:waleffe_lifetimes} (a) and (b) the average lifetimes for each box in the  attracting set for $R = 101$ and $R=102$, respectively. 
	
	\begin{figure}[!htb]
		\centering
		\begin{minipage}{0.4\textwidth}
			\includegraphics[width = \textwidth]{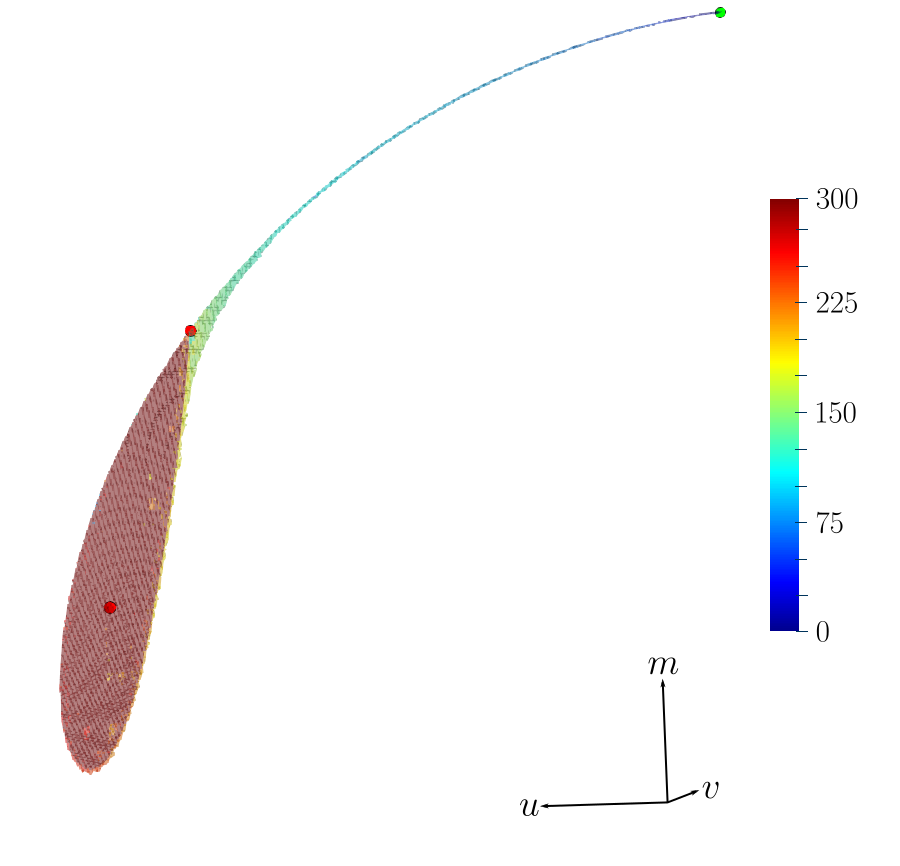}\\
			\centering \scriptsize{(a) $R = 101$}
		\end{minipage}
		\hspace{2em}
		\begin{minipage}{0.4\textwidth}
			\includegraphics[width = \textwidth]{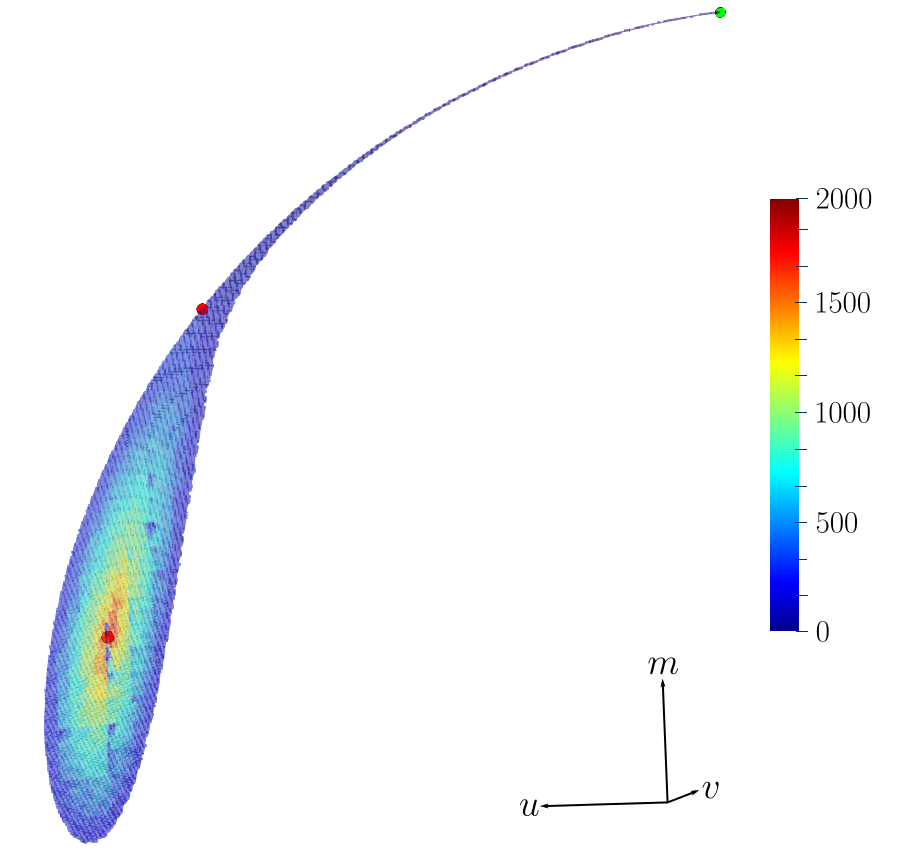}\\
			\centering \scriptsize{(b) $R = 102$}
		\end{minipage}
		\caption{Three-dimensional projection of the relative global attractor $A_Q^R$ of \eqref{eq:waleffe} for 
			$R = 101$ and $R=102$ with color-coded lifetimes.
			In (a), integration of test points was carried out for times up to $t = 300$. The long lifetimes in the droplet shaped region that ends at the saddle (red dot) indicate that all trajectories are attracted to the stable limit cycle. On the other side of the saddle, the connection to the laminar profile (green dot) has different colors, reflecting the different times it takes to reach the laminar fixed point. In (b) the integration was carried out up to $t = 2000$ to highlight the slow escape from the region of the limit cycle, which
			now has become unstable.}
		\label{fig:waleffe_lifetimes}
	\end{figure}
	
	For $R=101$ we see a strict separation in the attracting set by the edge state: all boxes between the edge state and the laminar state have finite lifetimes, as they are attracted towards the laminar profile. All boxes inside the lobe around the upper branch state have infinite lifetimes because they are attracted to the limit cycle and cannot return to the laminar profile. Here, test points in each box were only integrated for times up to $t = 300$: the choice of time is not essential, since for a closed attractor around the upper branch state no trajectory can return to the laminar flow. For $R=102$ the limit cycle is not stable anymore, and all points except for the fixed points and the limit cycle return to the laminar profile. Observe that the obtained box coverings do not only cover the laminar profile but also the unstable manifolds which is due to the fact that the relative global attractor contains all backward invariant sets.
	
	For $102 < R \leq 356$ the attractor develops a more complicated structure with folds and other features (Figure~\ref{fig:waleffe4Db} (a--d)), that arise from the projection from the four-dimensional state space to the three-dimensional image plane.
	
	\begin{figure}[H]
		\begin{center}
			\begin{minipage}{0.3275\textwidth}
				\includegraphics[width = \textwidth, height = \textwidth]{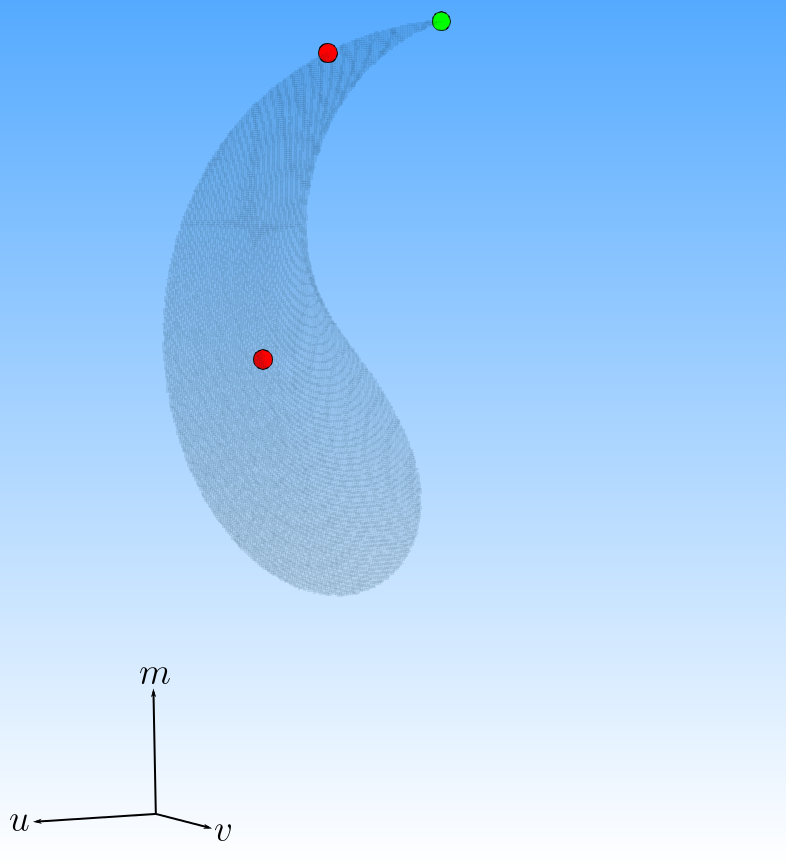}\\
				\centering \scriptsize{(a) $R = 139$}
			\end{minipage}
			\hfill
			\begin{minipage}{0.320\textwidth}
				\includegraphics[width = \textwidth, height = \textwidth]{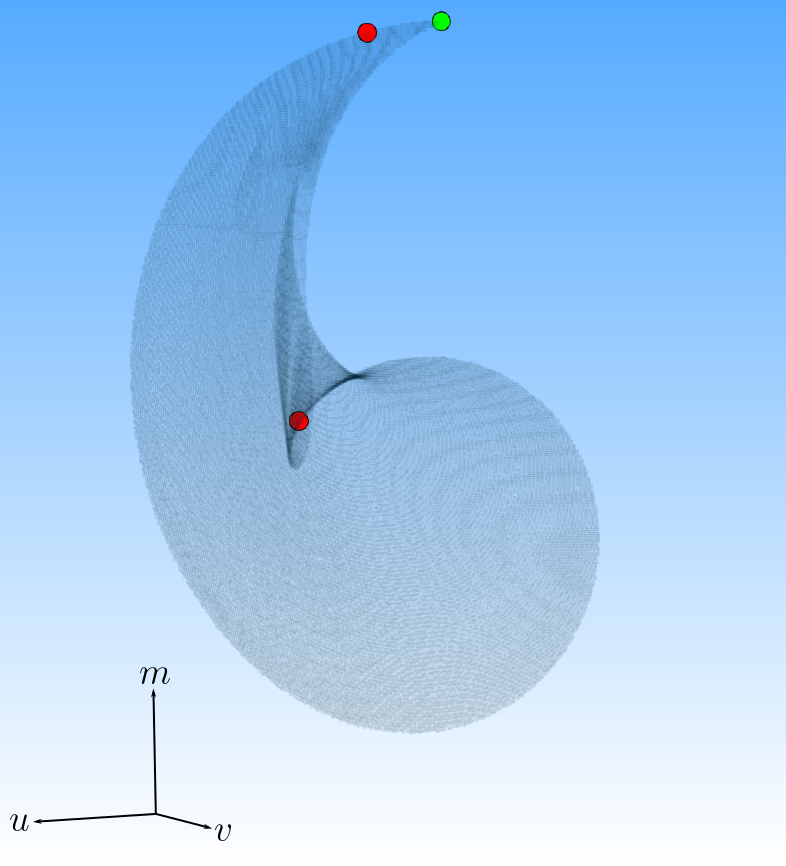}\\
				\centering \scriptsize{(b) $R = 200$}
			\end{minipage}
			\hfill
			\begin{minipage}{0.320\textwidth}
				\includegraphics[width = \textwidth, height = \textwidth]{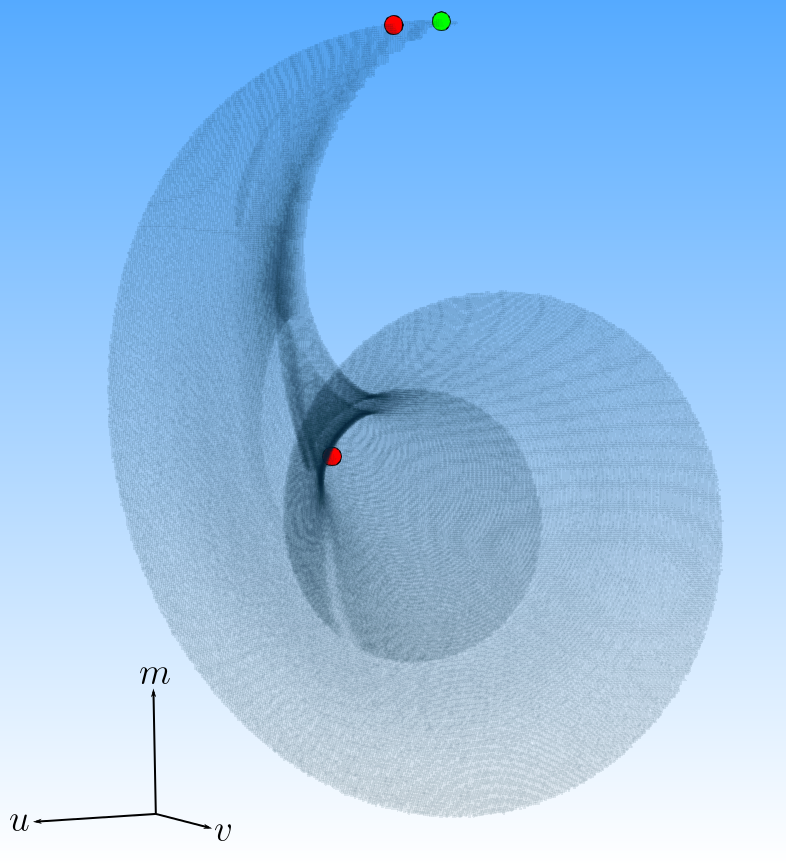}\\
				\centering \scriptsize{(c) $R = 300$}
			\end{minipage}\\
			\begin{minipage}{0.320\textwidth}
				\includegraphics[width = \textwidth, height = \textwidth]{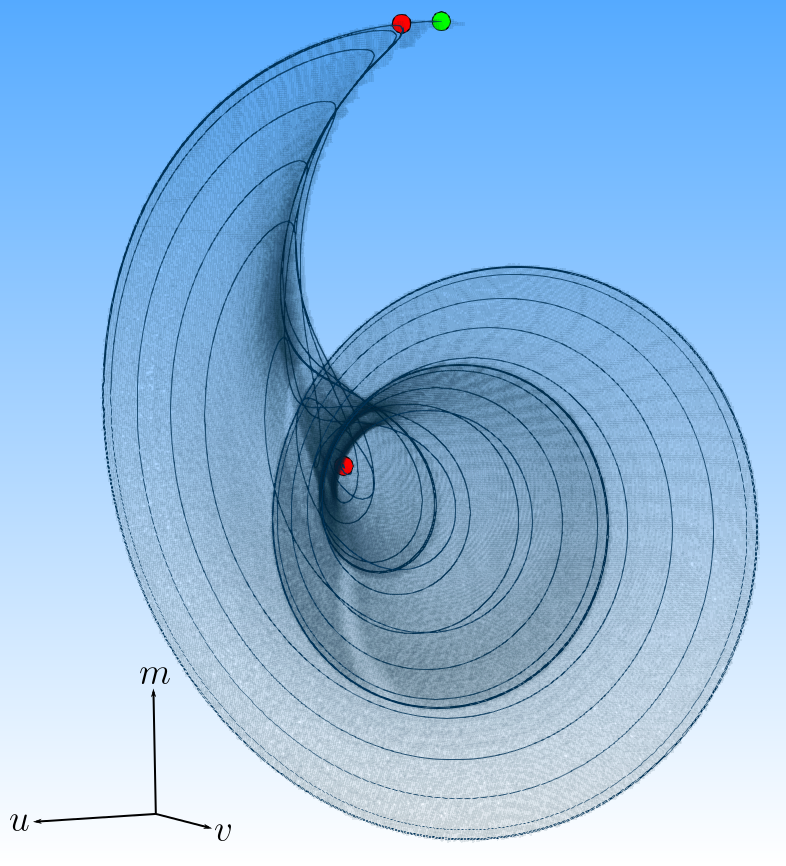}\\
				\centering \scriptsize{(d) $R = 356$}
			\end{minipage}
			\hfill
			\begin{minipage}{0.320\textwidth}
				\includegraphics[width = \textwidth, height = \textwidth]{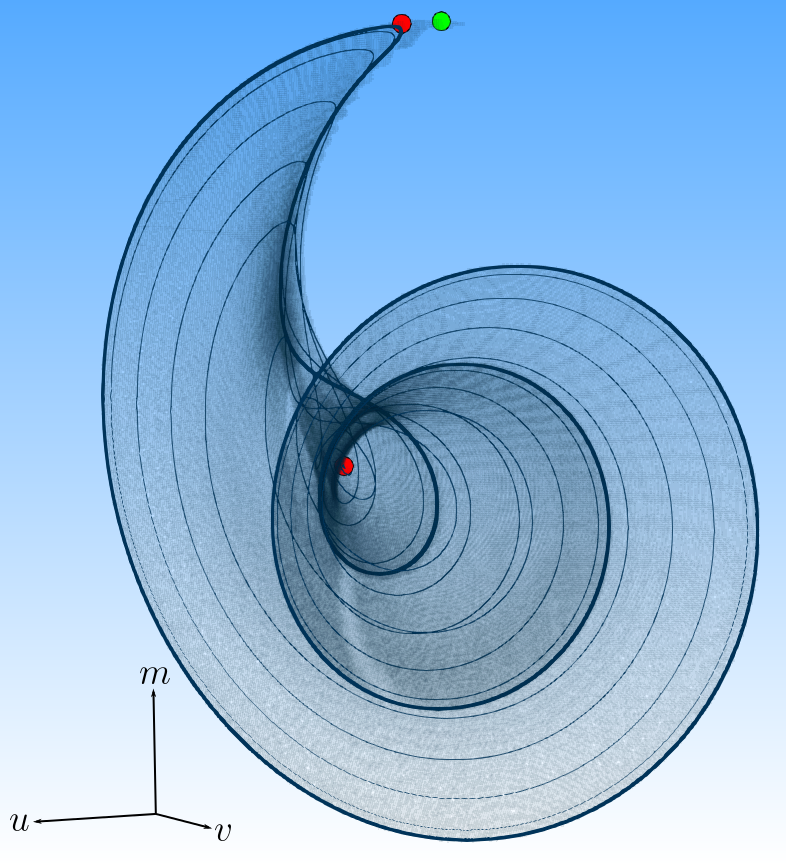}\\
				\centering \scriptsize{(e) $R = 357$}
			\end{minipage}
			\hfill
			\begin{minipage}{0.320\textwidth}
				\includegraphics[width = \textwidth, height = \textwidth]{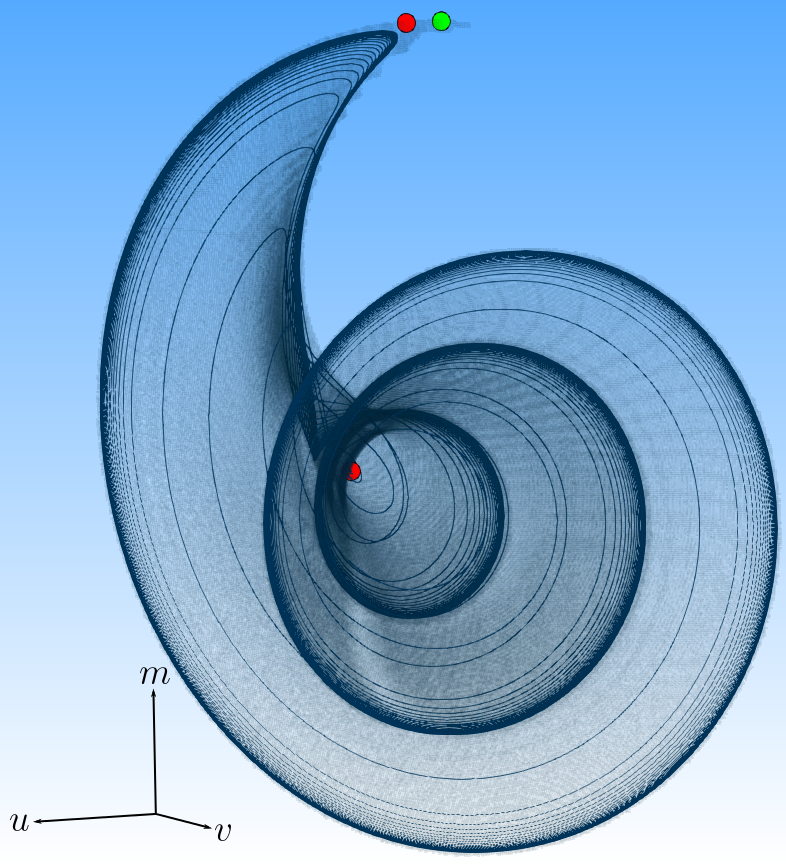}\\
				\centering \scriptsize{(f) $R = 400$}
			\end{minipage}
		\end{center}
		\caption{Three-dimensional projections of the relative global attractor $A_Q^{R}$ for the low dimensional model of a self-sustained flow \eqref{eq:waleffe}. Stable and unstable fixed points are shown in green and red, respectively. The green fixed point at the top of each figure depicts the laminar solution $(u,v,w,m) = (0,0,0,1)$ which remains stable for all values of $R$ (cf.~\cite{Wal95b}).}
		\label{fig:waleffe4Db}
	\end{figure}
	
	Finally, for $356 < R < 435$ another stable limit cycle appears (cf.~Figure~\ref{fig:waleffe4Db} (e) and (f)).

	\subsection{A nine-dimensional model for turbulent shear flows}
	As a second example we consider a Fourier based nine-dimensional model for 
	a turbulent shear flow with free slip boundaries and a sinusoidal base profile. A first version with 8 degrees of freedom was described by Waleffe \cite{Wal95a,Wal95b}, and extended to 9 modes in \cite{MFE04,MFE05}. The model is based on a Fourier expansion of the velocity field, and every component can be assigned to a specific flow field. In addition to the Reynolds number, the model has two spatial parameters describing the domain size in spanwise and downstream direction. The choice we consider here is the NBC domain (Nagata, Busse, Clever \cite{Nagata:1990ae,CB97}), i.e., a box with $L_z=2\pi$ in spanwise and $L_x=4\pi$ in downstream direction. For these parameters, the first state that appears	is a periodic orbit, at $R=89.76$. The periodic orbit is unstable from the onset, so that the dynamics is transient without the need for the occurrence of a crisis bifurcation. However, the lifetime increases very rapidly with Reynolds number \cite{MFE04}.
	
	Again, we choose the Reynolds number $R$ as our parameter of interest and compute the relative global attractor $A_Q^{R_j}$ for Reynolds numbers starting with an initial $R_0 = 200$ and then decreasing in steps of one, $R_{j+1} = R_{j} - 1$, ${j=0,\ldots,100}$. Due to the complexity of the dynamics in the nine-dimensional phase space it is hard to visualize the attracting set. An impression of the set may be obtained by three-dimensional projections onto the modes $(a_1,a_4,a_3)$ that correspond to $(m,u,v)$ in the four-dimensional model \eqref{eq:waleffe} and to the modes $(a_1,a_6,a_2)$ which represent the principal components of the set.
	
	\begin{figure}[H]
		\begin{center}
			\begin{minipage}{0.320\textwidth}
				\includegraphics[width = \textwidth]{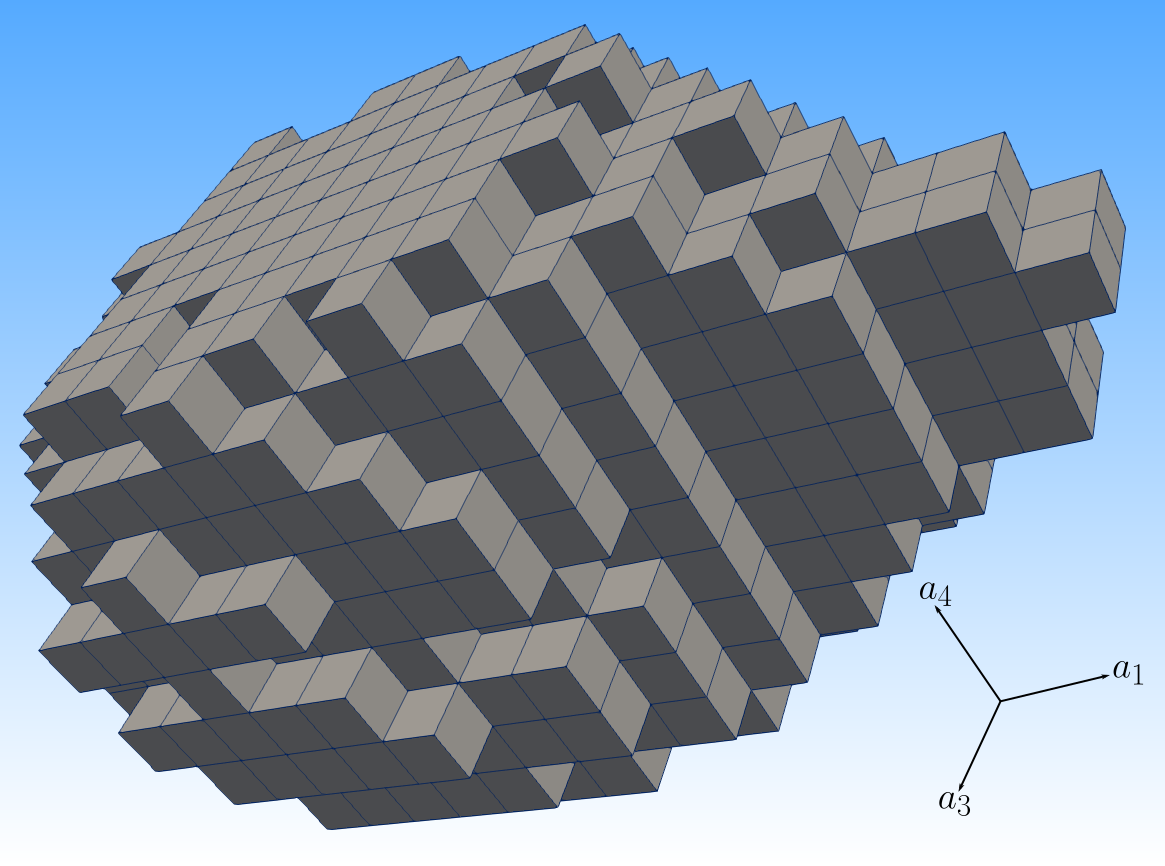}\\
			\end{minipage}
			\hfill
			\begin{minipage}{0.320\textwidth}
				\includegraphics[width = \textwidth]{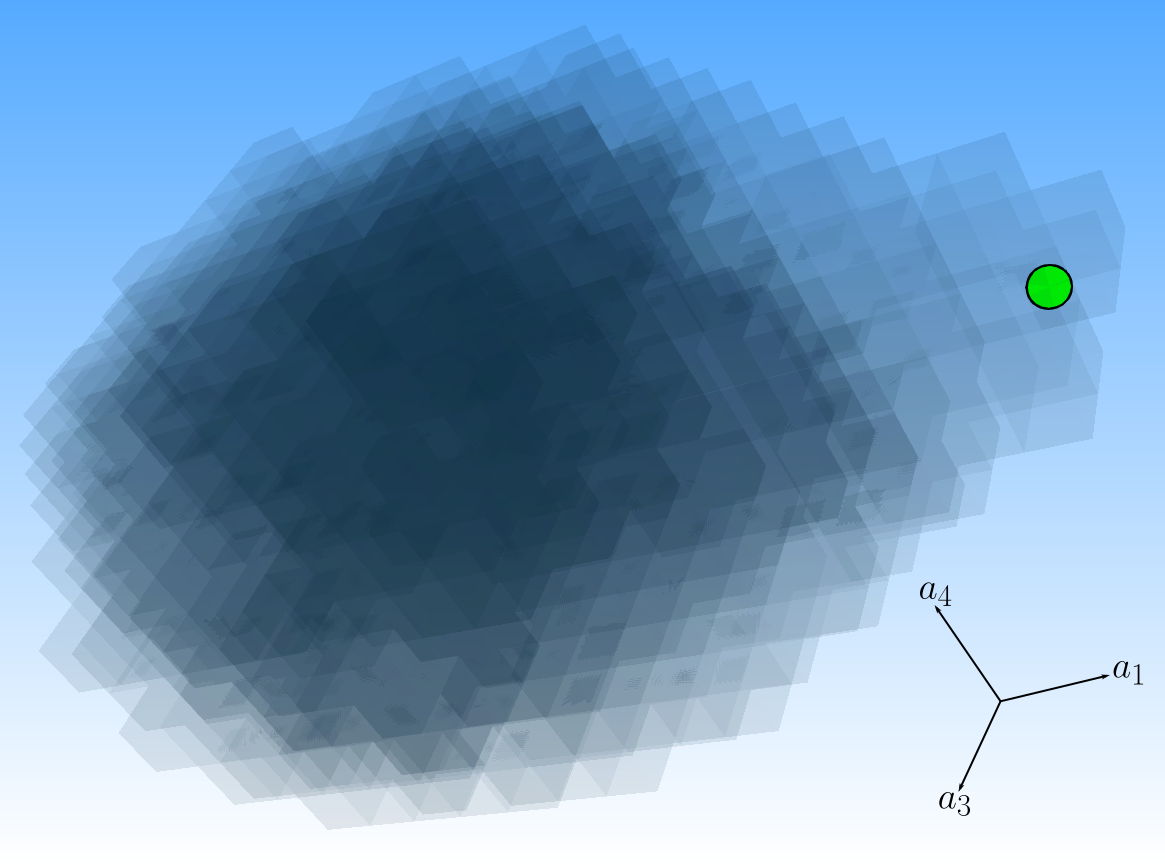}\\
			\end{minipage}
			\hfill
			\begin{minipage}{0.320\textwidth}
				\includegraphics[width = \textwidth]{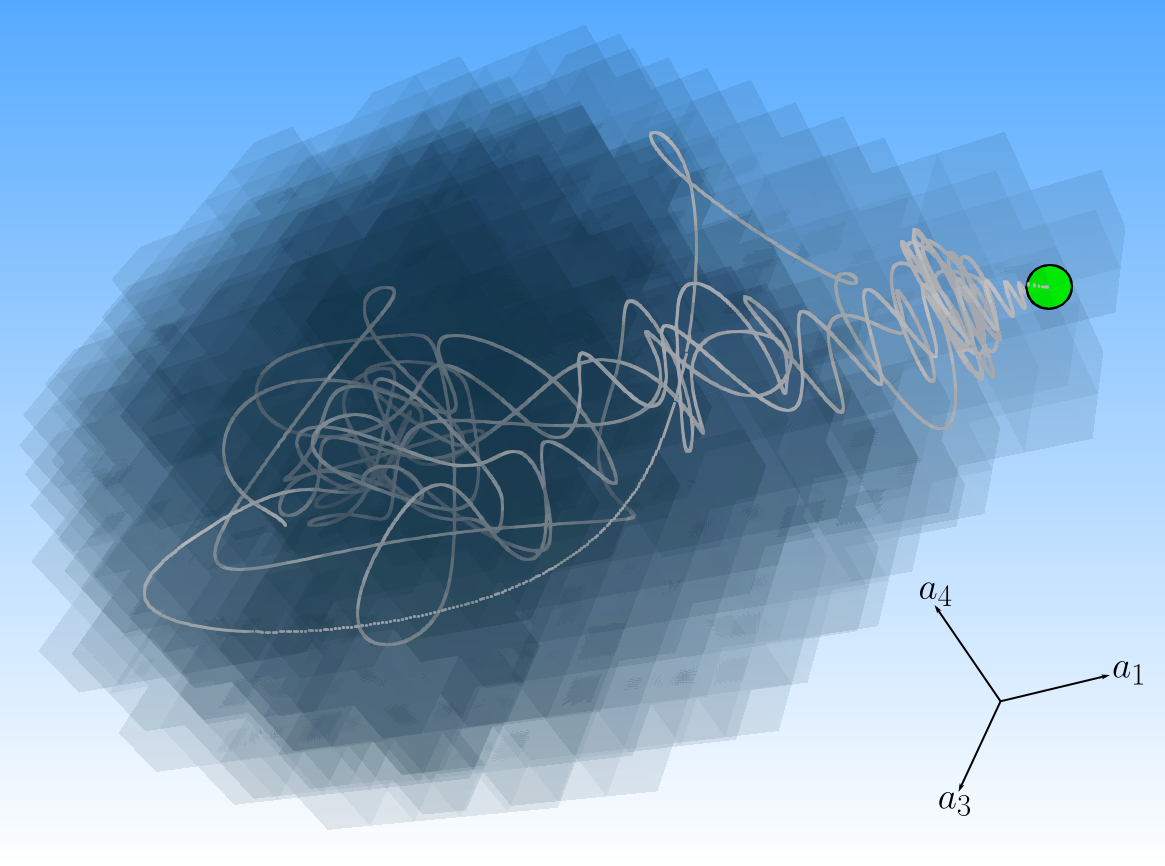}\\
			\end{minipage}\\
			\begin{minipage}{0.320\textwidth}
				\includegraphics[width = \textwidth]{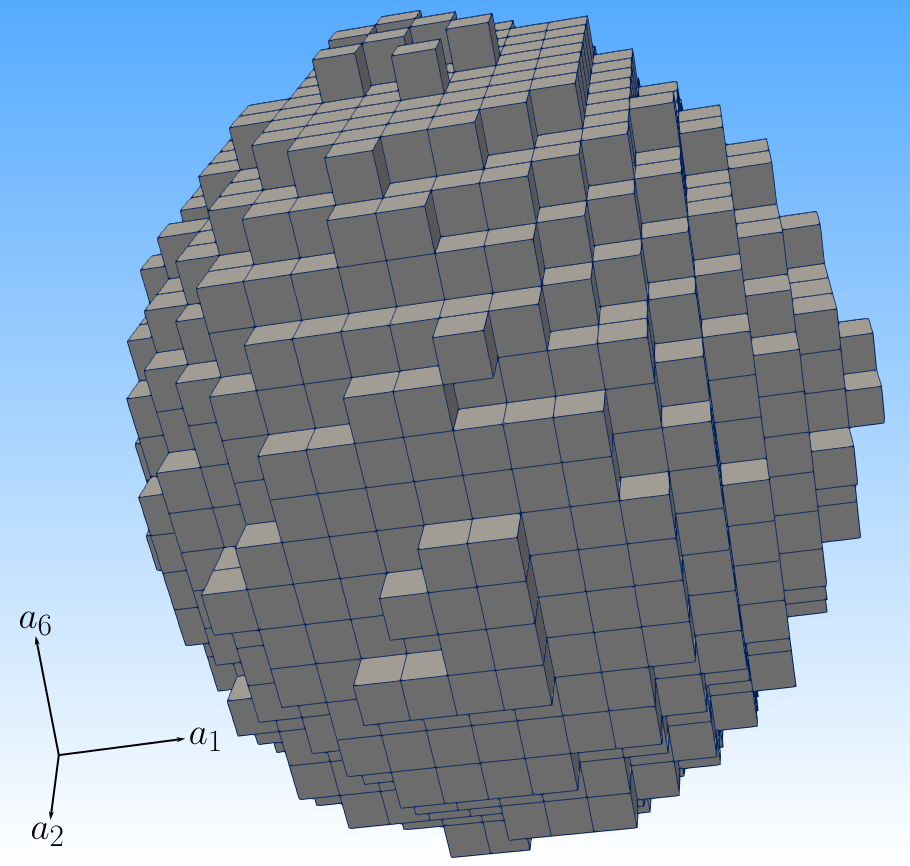}\\
				\centering \scriptsize{(a) Solid attractor}
			\end{minipage}
			\hfill
			\begin{minipage}{0.320\textwidth}
				\includegraphics[width = \textwidth]{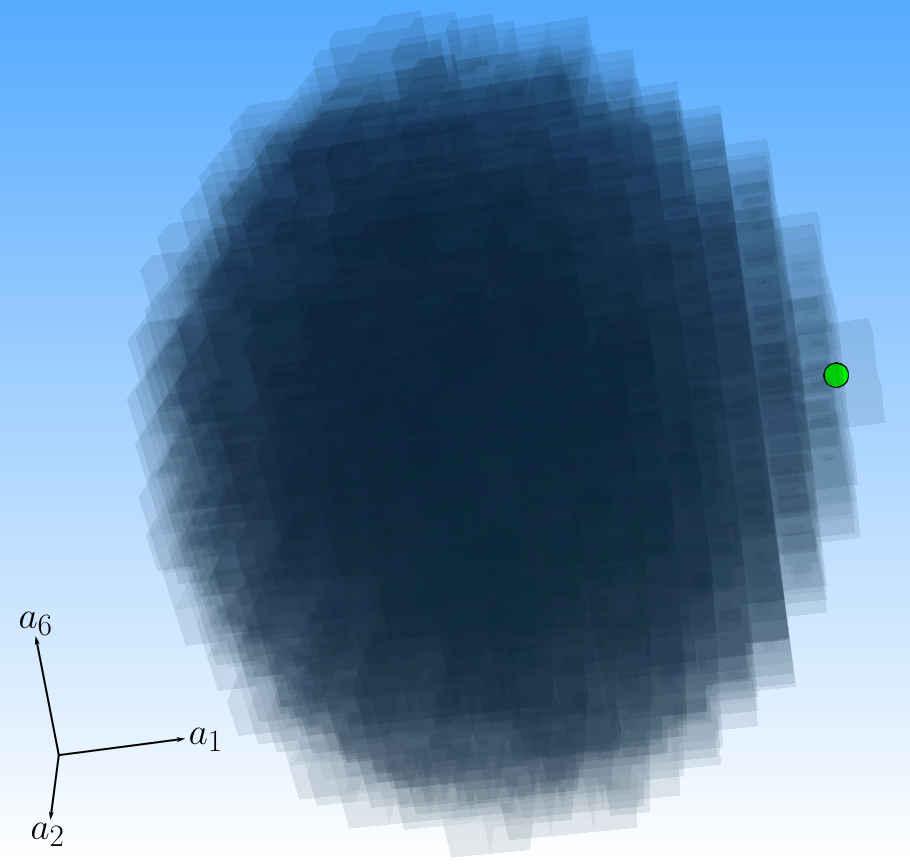}\\
				\centering \scriptsize{(b) Transparent boxes}
			\end{minipage}
			\hfill
			\begin{minipage}{0.320\textwidth}
				\includegraphics[width = \textwidth]{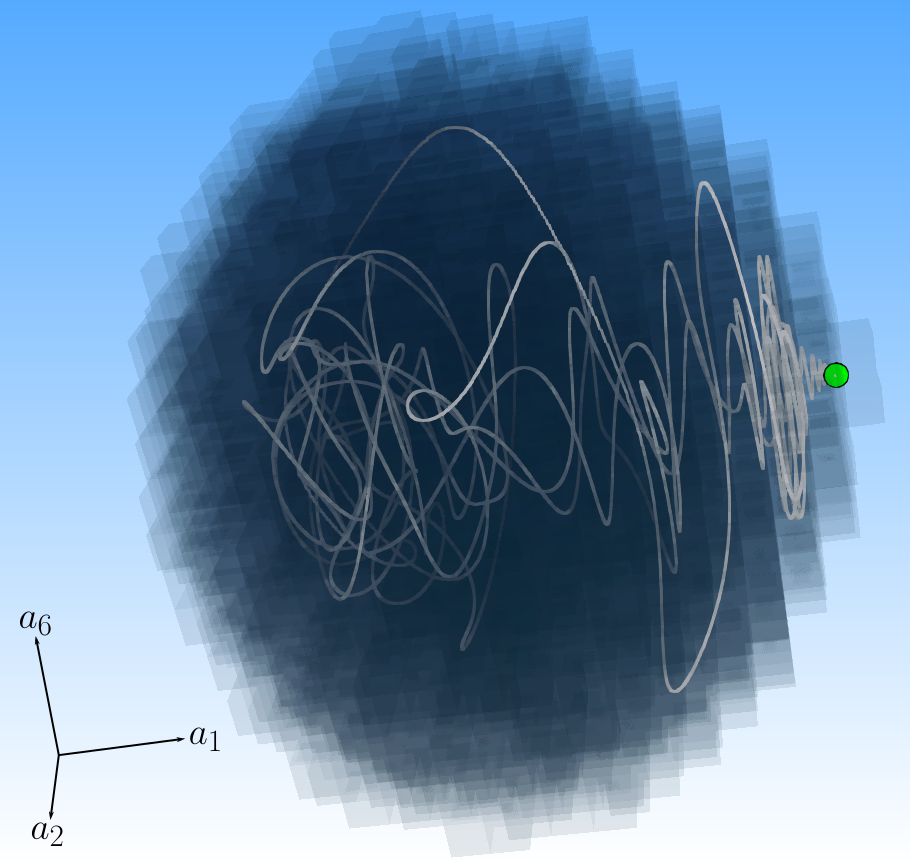}\\
				\centering \scriptsize{(c) Typical
					trajectory}
			\end{minipage}
		\end{center}
		\caption{Two three-dimensional projections of the relative global attractor $A_Q^{R}$ for the nine-dimensional model for turbulent shear flow for $R = 200$. 
			The green dot depicts the laminar solution $a_1 = 1, a_2=a_3=\ldots=a_9 = 0$ which is stable.	In (c) the attracting set is transparent so that a typical trajectory inside the attracting set can be	seen (it is the same trajectory in both projections).}
		\label{fig:Eckhardt9d}
	\end{figure}
	
	In Figure~\ref{fig:Eckhardt9d} we see that these projections of the attractor look like dense balls with no discernable interior structure. A measure of the increasing complexity of the dynamics is the dimension of the attracting set, here
	quantified by the box-counting dimension (cf.~\cite{Fal13}) of the box coverings obtained by Algorithm~\ref{alg:pathFollowing}. Figure~\ref{fig:bcd_9d} 
	shows that this dimension increases with Reynolds number. For $R=100$ the dimension of the attracting set is about $d = 2.6$ and increases to about $3.1$ for $R=200$ and most likely even higher values for $R > 200$.
	
	\begin{figure}[!htb]
		\centering
		\includegraphics[width=.6\textwidth]{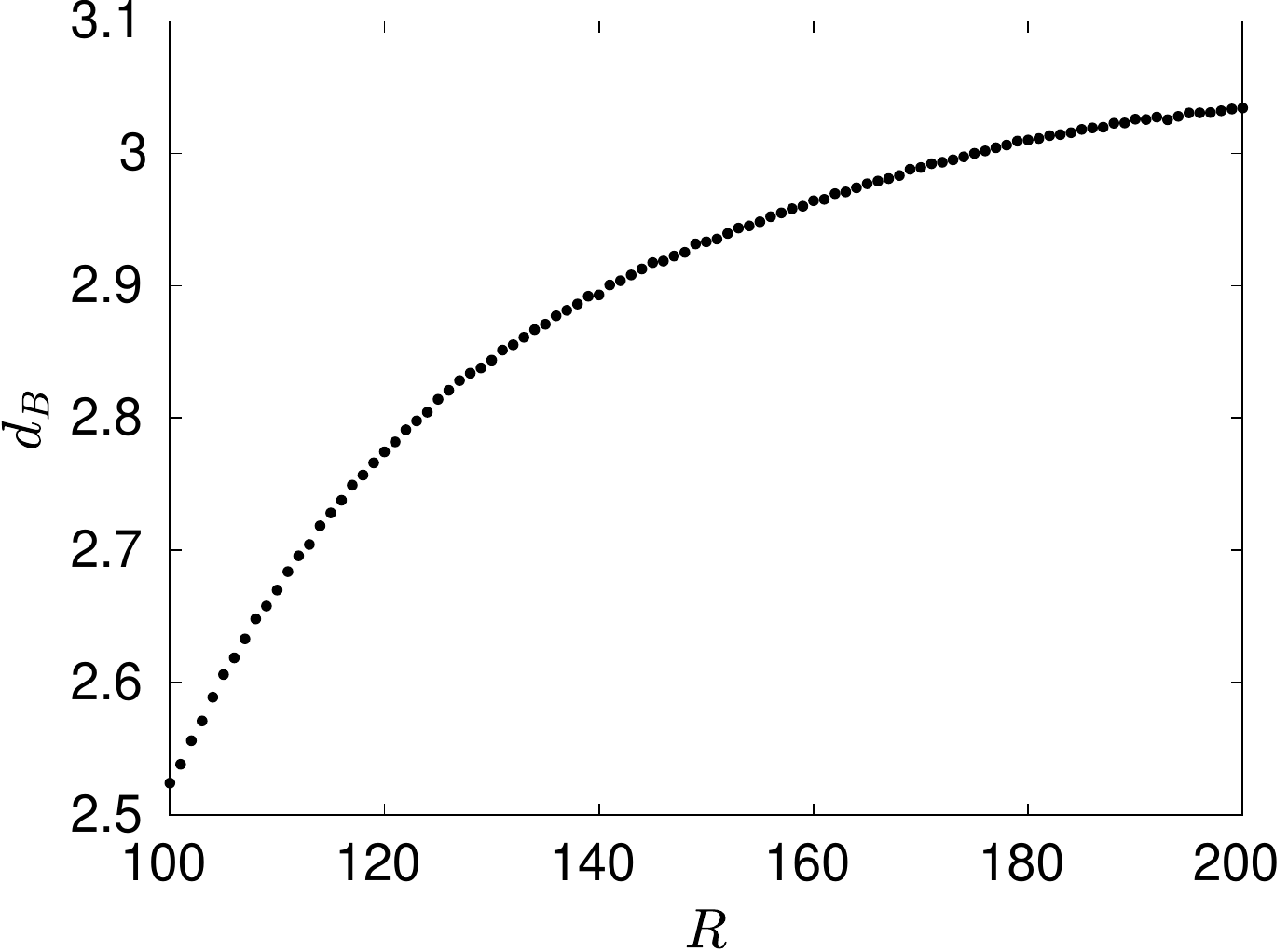}
		\caption{Box-counting dimension of the attracting set of the nine-dimensional model for turbulent shear flows depending on the Reynolds number $R$.}
		\label{fig:bcd_9d}
	\end{figure}
	
	\section{Conclusions}\label{sec:conclusion}
	The set-oriented path following method is a powerful tool for the numerical analysis of global bifurcations for attracting sets.
	With the method developed in this paper the actual covering of an attracting set
	for a parameter value $\lambda$ can be used as an initial covering for a parameter $\bar \lambda$ which is sufficiently close to $\lambda$.  Thus, we do not need to restart the subdivision method, thereby significantly reducing the overall computational time. However, this path following method is only defined for finite dimensional dynamical systems. 
	Therefore, it should be desirable to extend this work to the infinite dimensional setting, i.e., for the parameter dependent computation of attractors of infinite dimensional dynamical systems. In this context the results of \cite{DHZ16,ZDG18,Zie18} would have to be adapted.
	
	The results in this article show that set-oriented methods can also be used in moderate-dimensional settings, where they can provide insight into the dynamics of shear flows during the transition to turbulence in simple models. 
	The similarities in the behaviour of the low-dimensional models and fully resolved numerical simulations suggests that the global organization of phase space will be similar, even though such a study is computationally out of reach for high-dimensional attractors. 
	
	
	\section*{Acknowledgments} This work is supported by the Priority Programme SPP 1881 Turbulent Superstructures of the Deutsche Forschungsgemeinschaft.

\bibliographystyle{plain}
\bibliography{Path_following}

\begin{thebibliography}{10}

\bibitem{AG93}
E.~L. Allgower and K.~Georg.
\newblock Continuation and path following.
\newblock {\em Acta numerica}, 2:1--64, 1993.

\bibitem{AG97}
E.~L. Allgower and K.~Georg.
\newblock Numerical path following.
\newblock {\em Handbook of numerical analysis}, 5:3--207, 1997.

\bibitem{Avila:2013jq}
M.~Avila, F.~Mellibovsky, N.~Roland, and B.~Hof.
\newblock {Streamwise-Localized Solutions at the Onset of Turbulence in Pipe
  Flow}.
\newblock {\em Phys Rev Lett}, 110(22):224502, May 2013.

\bibitem{Chandrasekhar61}
S.~Chandrasekhar.
\newblock {\em Hydrodynamic and hydromagnetic stability}.
\newblock Oxford University Press, Oxford, 1961.

\bibitem{Chossat94}
P.~Chossat and G.~Iooss.
\newblock {\em The {C}ouette-{T}aylor {P}roblem}.
\newblock Springer-Verlag, 1994.

\bibitem{CB97}
RM~Clever and Fritz~H Busse.
\newblock Tertiary and quaternary solutions for plane couette flow.
\newblock {\em Journal of Fluid Mechanics}, 344:137--153, 1997.

\bibitem{DS13}
H.~Dankowicz and F.~Schilder.
\newblock {\em Recipes for continuation}, volume~11 of {\em Computational
  Science and Engineering}.
\newblock SIAM, 2013.

\bibitem{WGK+12}
V.~De~Witte, W.~Govaerts, Y.~A. Kuznetsov, and M.~Friedman.
\newblock Interactive initialization and continuation of homoclinic and
  heteroclinic orbits in matlab.
\newblock {\em ACM Transactions on Mathematical Software (TOMS)}, 38(3):18,
  2012.

\bibitem{DFJ01}
M.~Dellnitz, G.~Froyland, and O.~Junge.
\newblock The algorithms behind {GAIO} --- {S}et oriented numerical methods for
  dynamical systems.
\newblock In {\em Ergodic {T}heory, {A}nalysis, and {E}fficient {S}imulation of
  {D}ynamical {S}ystems}, pages 145--174. Springer-Verlag, 2001.

\bibitem{DGM+05}
M.~Dellnitz, K.~A. Grubits, J.~E. Marsden, K.~Padberg, and B.~Thiere.
\newblock Set oriented computation of transport rates in 3-degree of freedom
  systems: the {R}ydberg atom in crossed fields.
\newblock {\em Regular and {C}haotic {D}ynamics}, 10(2):173--192, 2005.

\bibitem{DHZ16}
M.~Dellnitz, M.~{Hessel-von Molo}, and A.~Ziessler.
\newblock On the computation of attractors for delay differential equations.
\newblock {\em Journal of Computational Dynamics}, 3(1):93--112, 2016.

\bibitem{DH96}
M.~Dellnitz and A.~Hohmann.
\newblock The {C}omputation of {U}nstable {M}anifolds using {S}ubdivision and
  {C}ontinuation.
\newblock In {\em Nonlinear {D}ynamical {S}ystems and {C}haos}, pages 449--459.
  Springer, 1996.

\bibitem{DH97}
M.~Dellnitz and A.~Hohmann.
\newblock A subdivision algorithm for the computation of unstable manifolds and
  global attractors.
\newblock {\em Numerische Mathematik}, 75:293--317, 1997.

\bibitem{DJ99}
M.~Dellnitz and O.~Junge.
\newblock On the approximation of complicated dynamical behavior.
\newblock {\em SIAM Journal on Numerical Analysis}, 36(2):491--515, 1999.

\bibitem{DJ05}
M.~Dellnitz and O.~Junge.
\newblock Set oriented numerical methods in space mission design.
\newblock {\em Modern {A}strodynamics}, 1:127--153, 2006.

\bibitem{DJK+05}
M.~Dellnitz, O.~Junge, W.~S. Koon, F.~Lekien, M.~W. Lo, J.~E. Marsden,
  K.~Padberg, R.~Preis, S.~D. Ross, and B.~Thiere.
\newblock Transport in dynamical astronomy and multibody problems.
\newblock {\em International Journal of Bifurcation and Chaos},
  15(03):699--727, 2005.

\bibitem{DJLMPPRT05}
M.~Dellnitz, O.~Junge, M.~Lo, J.~E. Marsden, K.~Padberg, R.~Preis, S.~Ross, and
  B.~Thiere.
\newblock Transport of {M}ars-crossing asteroids from the quasi-{H}ilda region.
\newblock {\em Physical {R}eview {L}etters}, 94(23):231102, 2005.

\bibitem{DKZ17}
M.~Dellnitz, S.~Klus, and A.~Ziessler.
\newblock A set-oriented numerical approach for dynamical systems with
  parameter uncertainty.
\newblock {\em SIAM Journal on Applied Dynamical Systems}, 16(1):120--138,
  2017.

\bibitem{DDO+99}
P.~Deuflhard, M.~Dellnitz, O.~Junge, and C.~Sch{\"u}tte.
\newblock Computation of essential molecular dynamics by subdivision
  techniques.
\newblock In {\em {C}omputational {M}olecular {D}ynamics: {C}hallenges,
  {M}ethods, {I}deas}, pages 98--115. Springer-Verlag, 1999.

\bibitem{DGK03}
A.~Dhooge, W.~Govaerts, and Y.~A. Kuznetsov.
\newblock Matcont: a matlab package for numerical bifurcation analysis of odes.
\newblock {\em ACM Transactions on Mathematical Software (TOMS)},
  29(2):141--164, 2003.

\bibitem{Doe81}
E.~J. Doedel.
\newblock Auto: A program for the automatic bifurcation analysis of autonomous
  systems.
\newblock {\em Congr. Numer}, 30:265--284, 1981.

\bibitem{DF89}
E.~J. Doedel and M.~J. Friedman.
\newblock Numerical computation of heteroclinic orbits.
\newblock {\em Journal of Computational and Applied Mathematics},
  26(1-2):155--170, 1989.

\bibitem{Eckhardt:2007ix}
B.~Eckhardt.
\newblock {Turbulence transition in pipe flow: some open questions}.
\newblock {\em Nonlinearity}, 21(1):T1--T11, December 2007.

\bibitem{Eckhardt:2007ka}
B.~Eckhardt, T.~M. Schneider, B.~Hof, and J.~Westerweel.
\newblock {Turbulence Transition in Pipe Flow}.
\newblock {\em Annu Rev Fluid Mech}, 39(1):447--468, January 2007.

\bibitem{Eckhardt:2018js}
Bruno Eckhardt.
\newblock {Transition to turbulence in shear flows}.
\newblock {\em Physica A}, 504:121--129, August 2018.

\bibitem{Fal13}
K.~Falconer.
\newblock {\em Fractal geometry: mathematical foundations and applications}.
\newblock John Wiley \& Sons, 2013.

\bibitem{FD03}
G.~Froyland and M.~Dellnitz.
\newblock Detecting and locating near-optimal almost invariant sets and cycles.
\newblock {\em SIAM Journal on Scientific Computing}, 24(6):1839--1863, 2003.

\bibitem{FHRSSG12}
G.~Froyland, C.~Horenkamp, V.~Rossi, N.~Santitissadeekorn, and A.~{Sen Gupta}.
\newblock Three-dimensional characterization and tracking of an {A}gulhas ring.
\newblock {\em Ocean Modelling}, 52-53:69--75, 2012.

\bibitem{FLS10}
Gary Froyland, Simon Lloyd, and Naratip Santitissadeekorn.
\newblock Coherent sets for nonautonomous dynamical systems.
\newblock {\em Physica D: Nonlinear Phenomena}, 239:1527--1541, 2010.

\bibitem{Gibson:2009kp}
J.~F. Gibson, J.~Halcrow, and P.~Cvitanovi{\'c}.
\newblock {Equilibrium and travelling-wave solutions of plane Couette flow}.
\newblock {\em J Fluid Mech}, 638:243, September 2009.

\bibitem{GS03}
M.~Golubitsky and I.~Stewart.
\newblock {\em The symmetry perspective: from equilibrium to chaos in phase
  space and physical space}, volume 200.
\newblock Springer Science \& Business Media, 2003.

\bibitem{GSS12}
M.~Golubitsky, I.~Stewart, and D.~G. Schaeffer.
\newblock {\em Singularities and groups in bifurcation theory}, volume~2.
\newblock Springer Science \& Business Media, 2012.

\bibitem{Grossmann:2000}
S.~Grossmann.
\newblock {The onset of shear flow turbulence}.
\newblock {\em Rev Mod Phys}, 72(2):603--618, April 2000.

\bibitem{GH83}
J.~Guckenheimer and P.~Holmes.
\newblock Local bifurcations.
\newblock In {\em Nonlinear oscillations, dynamical systems, and bifurcations
  of vector fields}, pages 117--165. Springer, 1983.

\bibitem{Halcrow:2009jx}
J.~Halcrow, J.~F. Gibson, P.~Cvitanovi{\'c}, and D.~Viswanath.
\newblock {Heteroclinic connections in plane Couette flow}.
\newblock {\em J Fluid Mech}, 621:365, February 2009.

\bibitem{HK12}
J.~K. Hale and H.~Ko{\c{c}}ak.
\newblock {\em Dynamics and bifurcations}, volume~3.
\newblock Springer Science \& Business Media, 2012.

\bibitem{HG89}
Jack~K Hale and Genevi{\`e}ve Raugel.
\newblock Lower semicontinuity of attractors of gradient systems and
  applications.
\newblock {\em Annali di Matematica pura ed applicata}, 154(1):281--326, 1989.

\bibitem{Hof:2006ab}
B.~Hof, J.~Westerweel, T.~M. Schneider, and B.~Eckhardt.
\newblock {Finite lifetime of turbulence in shear flows.}
\newblock {\em Nature}, 443(7107):59--62, September 2006.

\bibitem{Kreilos:2012bd}
T.~Kreilos and B.~Eckhardt.
\newblock {Periodic orbits near onset of chaos in plane Couette flow}.
\newblock {\em Chaos}, 22(4):047505, 2012.

\bibitem{Kreilos:2014ew}
T.~Kreilos, B.~Eckhardt, and T.~M. Schneider.
\newblock {Increasing Lifetimes and the Growing Saddles of Shear Flow
  Turbulence}.
\newblock {\em Phys Rev Lett}, 112(4):044503, January 2014.

\bibitem{Lo63}
E.~N. Lorenz.
\newblock Deterministic nonperiodic flow.
\newblock {\em Journal of the {A}tmospheric {S}ciences}, 20(2):130--141, 1963.

\bibitem{MFE04}
J.~Moehlis, H.~Faisst, and B.~Eckhardt.
\newblock A low-dimensional model for turbulent shear flows.
\newblock {\em New Journal of Physics}, 6(1):56, 2004.

\bibitem{MFE05}
J.~Moehlis, H.~Faisst, and B.~Eckhardt.
\newblock Periodic orbits and chaotic sets in a low-dimensional model for shear
  flows.
\newblock {\em SIAM Journal on Applied Dynamical Systems}, 4(2):352--376, 2005.

\bibitem{Nagata:1990ae}
M.~Nagata.
\newblock {Three-dimensional finite-amplitude solutions in plane Couette flow :
  bifurcation from infinity}.
\newblock {\em J Fluid Mech}, 217:519--527, August 1990.

\bibitem{Schmid99}
P.~J. Schmid and D.~S. Henningson.
\newblock {\em Stability and Transition of Shear Flows}.
\newblock Springer, New York, 1999.

\bibitem{SHD01}
C.~Sch\"utte, W.~Huisinga, and P.~Deuflhard.
\newblock Transfer operator approach to conformational dynamics in biomolecular
  systems.
\newblock In {\em Ergodic {T}heory, {A}nalysis, and {E}fficient {S}imulation of
  {D}ynamical {S}ystems}, pages 191--223. Springer-Verlag, 2001.

\bibitem{Wal95a}
F.~Waleffe.
\newblock Hydrodynamic stability and turbulence: Beyond transients to a
  self-sustaining process.
\newblock {\em Studies in applied mathematics}, 95(3):319--343, 1995.

\bibitem{Wal95b}
F.~Waleffe.
\newblock Transition in shear flows. nonlinear normality versus non-normal
  linearity.
\newblock {\em {Physics of Fluids}}, 7(12):3060--3066, 1995.

\bibitem{Wig13}
S.~Wiggins.
\newblock {\em Global bifurcations and chaos: analytical methods}, volume~73.
\newblock Springer Science \& Business Media, 2013.

\bibitem{Zammert:2015jg}
S.~Zammert and B.~Eckhardt.
\newblock {Crisis bifurcations in plane Poiseuille flow}.
\newblock {\em Phys Rev E}, 91(4):041003, April 2015.

\bibitem{Zie18}
A.~Ziessler.
\newblock {\em Analysis of Infinite Dimensional Dynamical Systems by
  Set-Oriented Numerics}.
\newblock PhD thesis, Paderborn University, 2018.

\bibitem{ZDG18}
A.~Ziessler, M.~Dellnitz, and R.~Gerlach.
\newblock The numerical computation of unstable manifolds for infinite
  dimensional dynamical systems by embedding techniques.
\newblock {\em arXiv preprint arXiv:1808.08787}, 2018.

\end{thebibliography}
\end{document}